\documentclass[11pt]{article}

\usepackage{amsmath,amssymb}

\usepackage{times}
\usepackage{bm}
\usepackage{natbib}
\usepackage{algorithm}
\usepackage{xargs}[2008/03/08]
\usepackage{amssymb}
\usepackage{mathrsfs}
\usepackage{graphicx}
\usepackage{rotating}
\usepackage[flushleft]{threeparttable}
\usepackage{multirow}
\usepackage{enumitem} 

\usepackage{star}
\usepackage{subcaption}

\usepackage[colorlinks,
linkcolor=blue,
anchorcolor=blue,
citecolor=blue
]{hyperref}

\def\T{{ \mathrm{\scriptscriptstyle T} }}

\def\##1\#{\begin{align}#1\end{align}}
\def\$#1\${\begin{align*}#1\end{align*}}

\newcommand{\rF}{\textnormal{F}}

\def\T{{ \mathrm{\scriptscriptstyle T} }} 

\newcommand{\Rom}[1]{\text{\uppercase\expandafter{\romannumeral #1\relax}}}

\def\T{{ \mathrm{\scriptscriptstyle T} }}
\newtheorem{lem}{Lemma}

\newtheorem{rmk}{Remark}

\newcommand{\R}{\mathbb{R}}

\newcommand{\alcomment}[1]{{\textcolor{red}{#1}}}
\newcommand{\Prob}{\textnormal{pr}}
\newcommand{\error}{P}

\newcommand{\Ex}{E}
\newcommand{\nmin}{n_{\textnormal{min}}}

\def\T{{ \mathrm{\scriptscriptstyle T} }}
\newcommand{\mumax}{\mu_{\textnormal{max}}}
\newcommand{\mudiff}{\mu_{\textnormal{diff}}}
\newcommand{\snr}{\textsc{SNR}}
\newcommand{\sigmax}{\sigma_{\textnormal{max}}}
\def\mtr{\mathop{\rm tr}}

\def\dist{\textrm{dist}}
\def\rF{\textnormal{F}}

\renewcommand{\opnorm}[2]{|  |  | #1 |  |  |_{{#2}}}

\usepackage{txfonts}

\usepackage{geometry}
\begin{document}

\title{ \LARGE Exact Cluster Recovery via Classical Multidimensional Scaling}     


\author{Anna Little\footnotemark[1], ~Yuying Xie\thanks{Department of Computational Mathematics, Science and Engineering, Michigan State University; Email: \texttt{\{littl119,xyy\}@egr.msu.edu}.}, ~and~Qiang Sun\thanks{Department of Statistical Sciences, University of Toronto; Email:\texttt{qsun@utstat.toronto.edu}. }}


\date{ }

\maketitle


\begin{abstract}
Classical multidimensional scaling is an important dimension reduction technique. Yet few theoretical results characterizing its statistical performance exist. This paper provides a theoretical framework for analyzing the quality of embedded samples produced by classical multidimensional scaling. This lays the foundation for various downstream statistical analyses,  and we focus on  clustering  noisy data. Our results provide scaling conditions on the sample size, ambient dimensionality, between-class distance, and noise level under which classical multidimensional scaling followed by a {distance-based} clustering algorithm can recover the cluster labels of all samples with high probability. Numerical simulations confirm these scaling conditions are near-sharp. Applications to both human genomics data and natural language data lend strong support to the methodology and theory. 

\end{abstract}

\section{Motivation and Background}

Embedding high dimensional data into a lower dimensional space is essential in many applications because the data pattern  is {then}  easier to visualize and  study. Multidimensional scaling (MDS) achieves this  by searching for a few coordinates to preserve pairwise distances. This approach originated with \cite{young1938discussion} but was fully developed by \cite{torgerson1952multidimensional}, and further popularized by \cite{gower1966some}. {Interestingly,} the method  does not require access to the original data coordinates. This  property leads to its popularity in many applications such as bioinformatics \citep{tzeng2008multidimensional}, psychology \citep{rau2016model, carroll1970analysis}, economics \citep{machado2015analysis}, network science \citep{chung2008graphical} and pattern recognition,  {where often} only pairwise distances {or correlations} are available. 
MDS and its variants  are also important components in {many} nonlinear dimension reduction {algorithms,} {e.g.} Isomap \citep{tenenbaum2000global} and  local multidimensional scaling \citep{venna2006local, chen2009local}. When pairwise distances are noisy, a trace norm penalty can be incorporated  \citep{zhang2016distance, negahban2011estimation, lu2005framework, yuan2007dimension}. We refer readers to  \cite{borg2005modern}  for a systematic overview.

Despite the {prevelance} of MDS,  few theoretical results characterizing its statistical performance under randomness exist \citep{fan2018principal}. The literature does not offer a systematic treatment on the influence of ambient noise {and data dimension} on the quality of the MDS embedding \citep{peterfreund2018multidimensional}. This paper provides a theoretical framework to study the quality of the {MDS} embedded samples {under a probabilistic model}. Specifically, we establish {an entrywise} bound for the embedding errors {produced by classical multidimensional scaling} (CMDS).  This lays the foundation for analyzing various downstream procedures. {We focus on clustering noisy data. Specifically, we analyze when the class labels can be exactly recovered by first computing the CMDS embedding and then applying a distance-based clustering algorithm to the embedded samples,}  where exact recovery means all cluster labels are inferred correctly.  There are several advantages of this two-step  {approach}, referred to as the CMDS and clustering procedure. {First,  it overcomes} the curse of dimensionality by reducing the {dimension before clustering}. Second, it does not require access to data coordinates (for example, see the People data sets in Section \ref{subsec:realworlddata}). Theoretically, we provide {general} scaling conditions {relating the signal-to-noise ratio ($\snr$), sample size, and ambient dimension which ensure} 
 exact recovery with high probability {for} the CMDS and clustering {procedure}. {We then illustrate the strong influence of the ambient dimension by deriving specific scaling conditions for the low, moderate, and high-dimensional regimes.}   Both simulations and real data studies support our theoretical findings. To the best of our knowledge, we are the first to provide sharp scaling conditions for exact cluster detection using the CMDS and clustering procedure, {although some recent works have considered the quality of the CMDS embedding under probabilistic models}.

Recently, \cite{li2018central} derived a central limit theorem characterizing the deviation of the CMDS embedding when the distance matrix is perturbed. Unlike their work, we do not assume the error matrix has independent entries, an assumption {that} often fails for distance matrices obtained from  data, and we derive non-asymptotic error bounds for the  CMDS {embedding}. Also, the results in \cite{li2018central} involve an unknown rotation matrix, and are thus less applicable for downstream data analysis. {In addition, \cite{peterfreund2018multidimensional}} propose procedures for {shrinking the {CMDS} eigenvalues to  improve the embedding quality under the spectral norm.}  In contrast,  {our}  loss function {is} based on an entrywise norm.  Furthermore, we {derive}  results for exact cluster recovery, {consider a more general noise model}, and provide a finite sample rather than asymptotic analysis.

\subsection{Review of Classical Multidimensional Scaling}
We briefly review the procedure of CMDS. Given the distance matrix $D \in \mathbb{R}^{N \times N}$ of $N$ data points  with the $(i,j)$th element being the Euclidean distance between the $i$th and $j$th samples, CMDS first obtains a matrix $B$ by applying double centering to the matrix of squared distances. Specifically, $B=-\frac{1}{2}JD^{(2)}J$ where {$D^{(2)}_{i,j} = D_{i,j}^2$}, $J=I_N-11^\T/N$ is the centering matrix, and $1\in \mathbb{R}^N$ is the column vector of all 1's.  Let $X=(x_1,\ldots, x_N)^\T  \in \mathbb{R}^{N \times d}$ be the possibly unobserved data matrix. 
It can be shown
\begin{align*}
B = -\frac{1}{2} JD^{(2)}J = (X- 1\hat{\mu}^\T )(X-1\hat{\mu}^\T )^\T {=(JX)(JX)^\T},
\end{align*}
where $\hat{\mu} = \sum_{i=1}^N x_i/N$ is the empirical mean of the data. See \cite{borg2005modern} for details. 
The coordinates of an $r$-dimensional embedding are then given by the rows of $Y=V_r \Lambda^{1/2}_r$, where  $\Lambda_r = \text{diag}(\lambda_1, \ldots, \lambda_r) \in \mathbb{R}^{r \times r}$ with $\lambda_i$ being the $i$th largest eigenvalue of $B$ and $V_r = (v_1, \ldots, v_r) \in \mathbb{R}^{N \times r}$  {with $v_r$ being} the  eigenvector {corresponding to $\lambda_r$}. Namely, each unobserved data point $x_i$ is mapped to  the $i$th row of $V_r\Lambda_r^{1/2}$. Mathematically, this embedding tries to  preserve pairwise distances  in $r$ dimensions by minimizing
{
	\begin{align*}
	\min_{Y \in \mathbb{R}^{N \times r}}\sum_{i,j=1}^N (D_{i,j}^X-D_{i,j}^Y)^2,
	\end{align*}
}
where $D^X, D^Y$ are the distance matrices corresponding to $X$ and $Y$, respectively.

We demonstrate this procedure on a toy example, by considering $5$ Gaussian distributions with means $\mu_k \in \mathbb{R}^{1000}$  for $k=1,\ldots, 5$, and a common covariance matrix {$\Sigma = 0.3I_{1000}$}.  The first two coordinates of $\mu_k$'s are $(0, 0), (1, 1), (1, -1), (-1, 1),$ and $(-1, -1)$ respectively,  while the rest of the coordinates are $0$s.  We sampled $200$ data points from each Gaussian distribution.  
The second plot in {Figure} \ref{fig:ExampleMDSEmbedding} shows that the CMDS embedding with only two dimensions can cluster the data into five clusters with high accuracy (97.8\%) while there is no visible separation using a randomly selected pair of original coordinates, as indicated by the left plot, {which} indicates the undesirable performances of random projection {\citep{vempala2005random}}.  The third plot in {Figure} \ref{fig:ExampleMDSEmbedding} shows the first two data coordinates colored by the labels returned by hierarchical clustering with Ward's method \citep{ward1963hierarchical} (accuracy 64\%); hierarchical clustering {performed} poorly due to the high-dimensional noise. The last plot shows the clustering accuracy of $k$-means on the CMDS embedding as a function of the embedding dimension $r$. The accuracy is high for small $r$ but declines as the dimension {increases}, illustrating the advantage of performing {dimension} reduction prior to clustering. 

\begin{figure}[t]
	\captionsetup[subfigure]{justification=centering}
	\centering
	\begin{subfigure}{.23\textwidth}
	\centering
		\includegraphics[width=\textwidth,keepaspectratio]{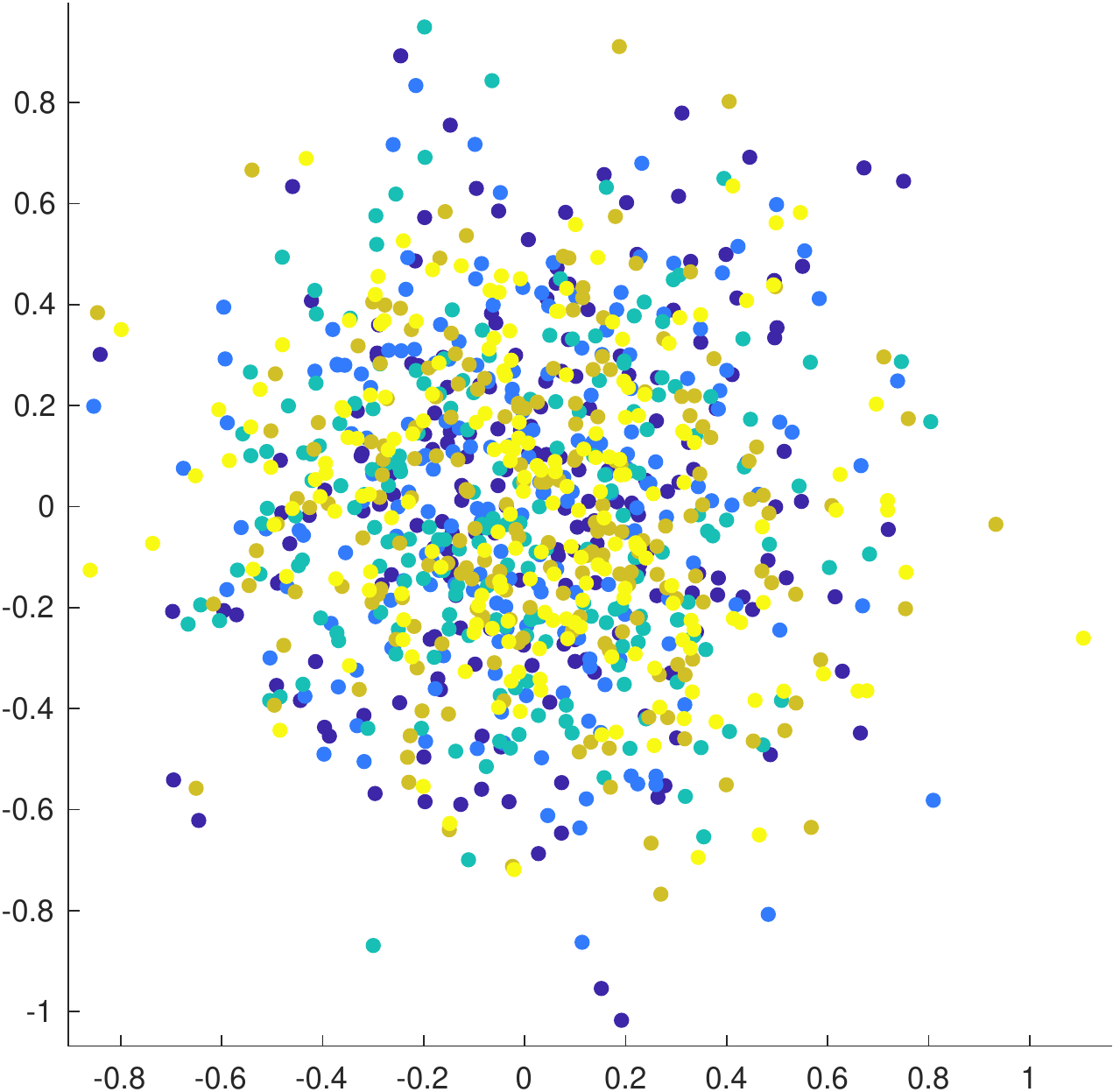}	
		\caption{\footnotesize Random coordinates}
		\label{fig:RandomDataCoordinates}
	\end{subfigure}	
	\hfill
	\begin{subfigure}{.23\textwidth}	
	\centering
		\includegraphics[width=\textwidth,keepaspectratio]{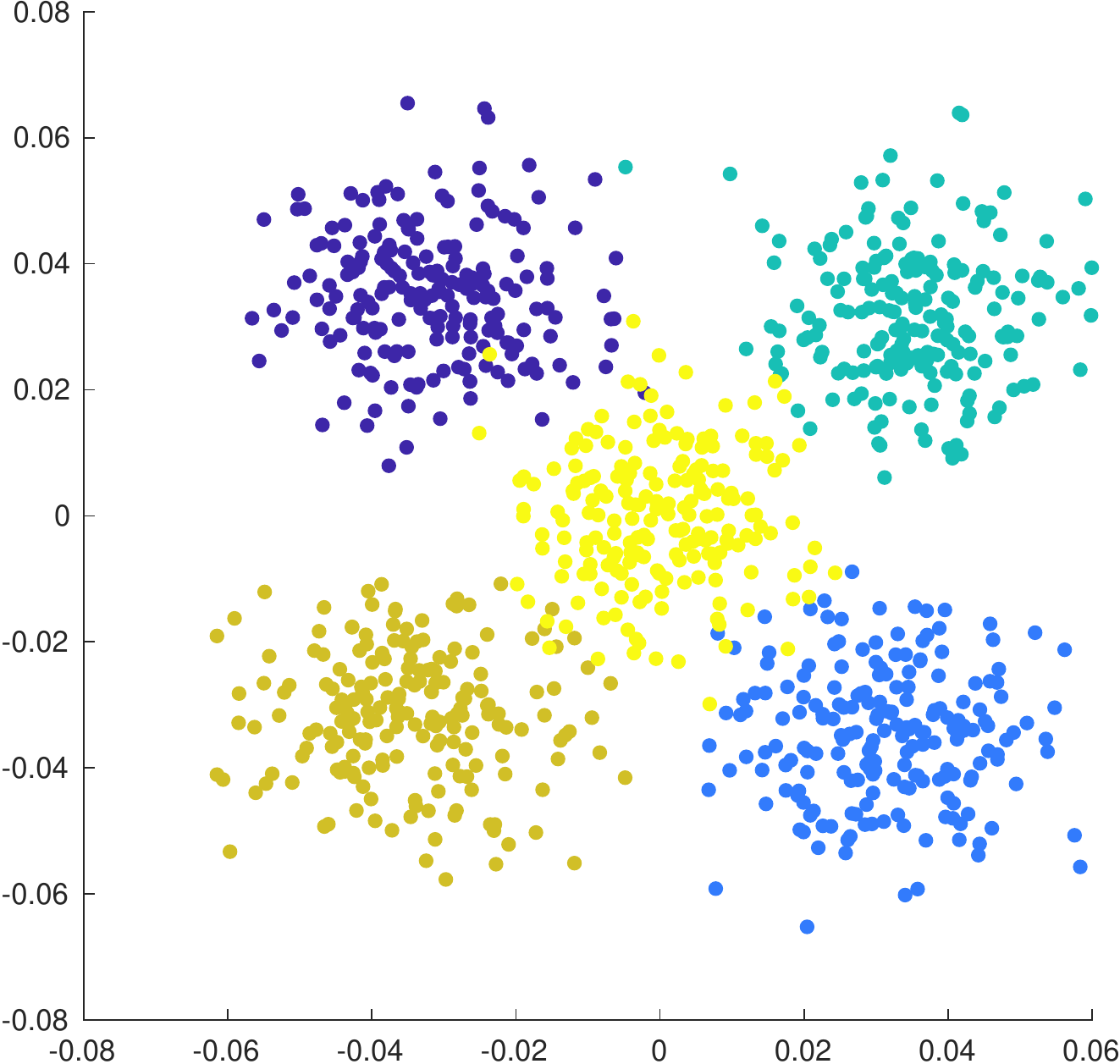}
		\caption{\footnotesize MDS embedding}
		\label{fig:MDSembedding}
	\end{subfigure}
	\hfill
	\begin{subfigure}{.23\textwidth}
	\centering
		\includegraphics[width=\textwidth,keepaspectratio]{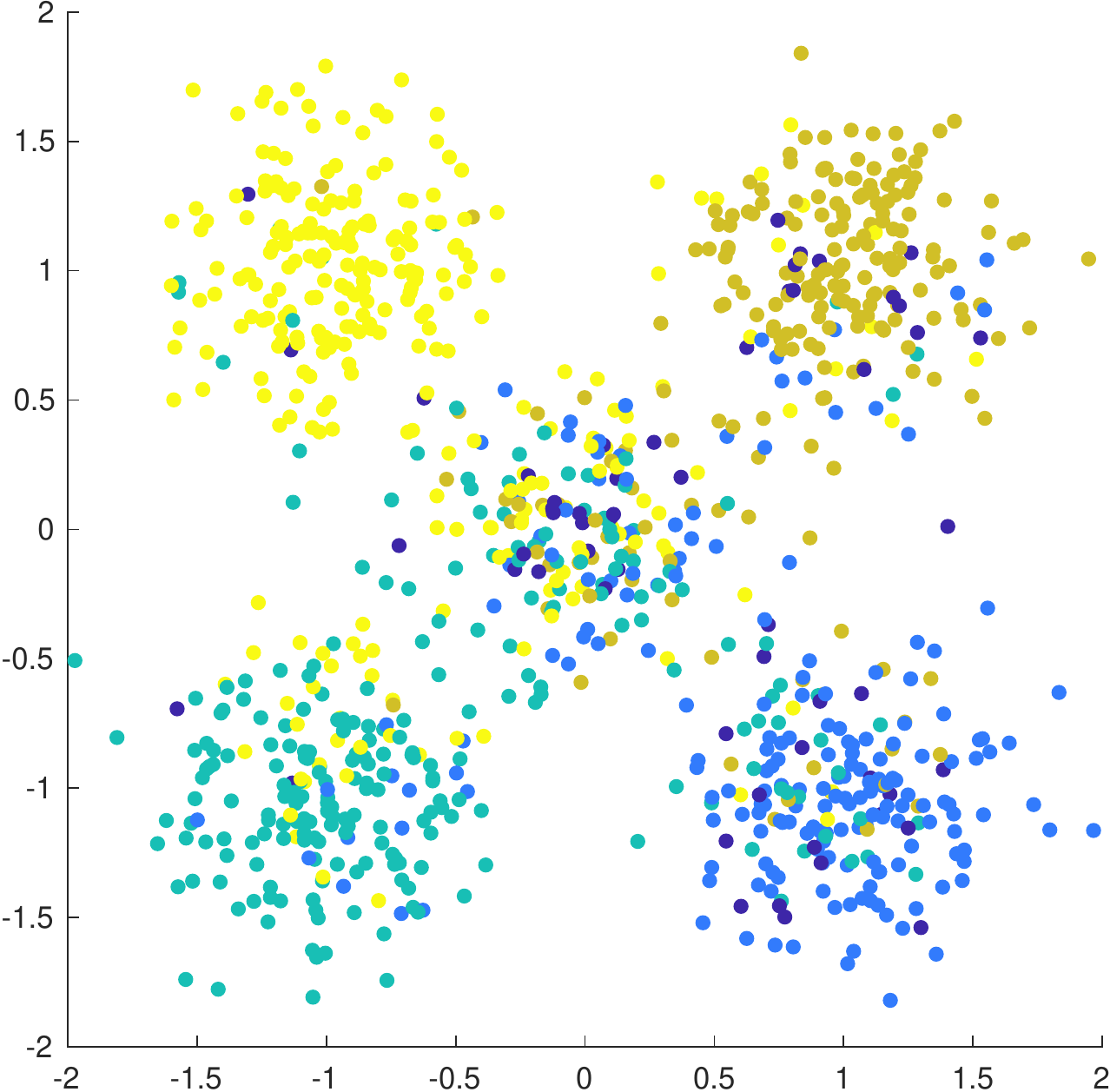}
		\caption{\footnotesize Hierarchical clusters}
		\label{fig:HierarchicalClustering}
	\end{subfigure}	 
	\hfill
	\begin{subfigure}{.23\textwidth}
	\centering
	\includegraphics[width=\textwidth,keepaspectratio]{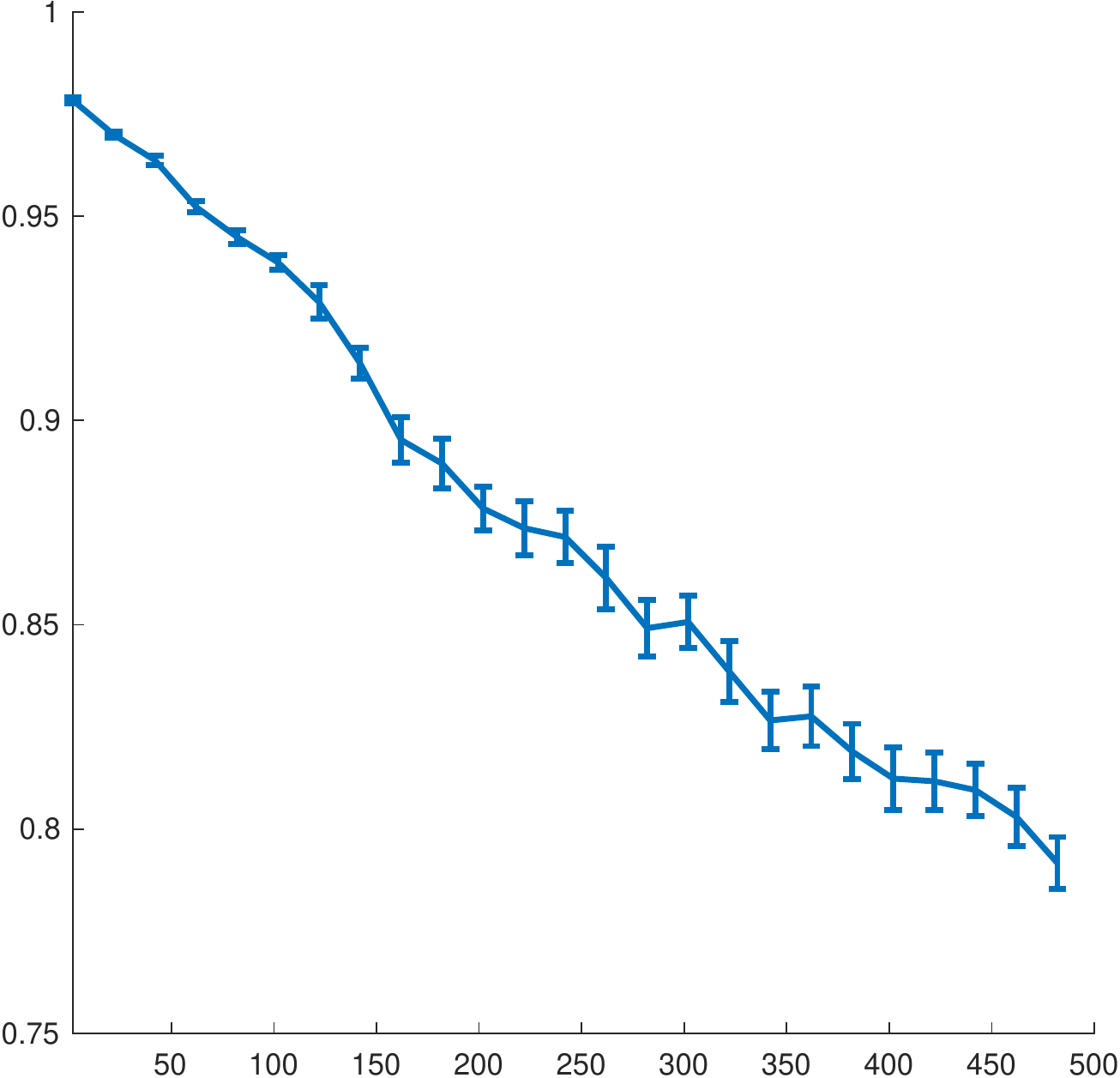}
	\caption{\footnotesize Clustering accuracy}
	\label{fig:AccuracyByDim}
	\end{subfigure}	 
	\caption{Figure \ref{fig:RandomDataCoordinates} shows two randomly selected data coordinates. Figure \ref{fig:MDSembedding} shows first two CMDS coordinates,
	colored by cluster identity. Figure \ref{fig:HierarchicalClustering} shows the first two data coordinates, colored by the labels returned by hierarchical clustering. Figure \ref{fig:AccuracyByDim} shows accuracy of $k$-means clustering on CMDS embedding as a function of embedding dimension $r$.}
	\label{fig:ExampleMDSEmbedding}
\end{figure}

\subsection{Notation}
We summarize the notation used throughout the paper. For any matrix $A = A_{i,j}$, the spectral norm is denoted by $\norm{A}_2$ and the Frobenius norm by $\| A\|_F$. Furthermore, $\norm{A}_{\max}=\max_{i,j}|A_{i,j}|$, $\norm{A}_{\infty} = \max_i \sum_{j} |A_{i,j}|$, $\norm{A}_{1} = \max_j \sum_{i} |A_{i,j}|$, and $A_{i\cdot}$ {is} the $i$th row of $A$.  For any vector $v$, $\norm{v}_2$ denotes the Euclidean norm and $\norm{v}_\infty = \max_i |v(i)|.$ Let $f(n)$ and $g(n)$ be functions that map integer $n$ to a positive real number, and let $c,b$ be positive constants and $n_0$ an integer. Then $f(n) = O(g(n))$ if $f(n) \leq c g(n)$ for all $n > n_0$; $f(n) = \Omega(g(n))$ if $f(n) \geq b g(n)$ for all $n > n_0$; and  $f(n) = \Theta(g(n))$ if $f(n) = O(g(n))$ and $f(n) = \Omega(g(n))$. Sometimes $f(n)\lesssim g(n)$ and $f(n)\gtrsim g(n)$ are used to denote $f(n)=O(g(n))$ and $f(n)=\Omega(g(n))$ respectively.

\section{Model Formulation}\label{sec:ModelFormulation}

{In this section we formulate precisely the model studied in this article. Section \ref{sec:DataModel} defines the noisy data model, Section \ref{sec:CMDSembedding} discusses how the CMDS embedding of the noisy data can be viewed as a perturbation of an ideal embedding, and Section \ref{sec:ExactRecovery} defines exact cluster recovery from the CMDS embedding.}  

\subsection{A Probabilistic Model} \label{sec:DataModel}

To investigate the theoretical properties of the CMDS and clustering procedure, we assume the unobserved data are drawn from $k$ clusters and model them as
\[
	x_i = \mu_{\ell_i} + \epsilon_i, \ \ 1 \leq i \leq N,
\]
{where $x_i \in \mathbb{R}^d$ is the $i$th sample, $\ell_i\in [k]:= \{1,2,\ldots,k\}$ is the cluster label of $x_i$, and $\epsilon_i$'s  are independently distributed   mean zero sub-Gaussian random errors. We say  $\epsilon \in\mathbb{R}^d$ is a sub-Gaussian random vector if the sub-Gaussian norm $\norm{\epsilon}_{\psi_2}$ is finite, where $\norm{\epsilon}_{\psi_2} := \max_{u \in S^{d-1}} \inf_t\{t > 0: \Ex[\exp(- (u^\T \eta)^2/t^2)] \leq 2\}.$ Suppose the $j$th cluster has $n_j$ independently distributed samples, and for simplicity, we assume $n_1\mu_1+ \cdots +n_k\mu_k=0$. Writing the model into a matrix form, we have
\$ 
	X_{N \times d} = M_{N \times d} + H_{N \times d},		
\$
where $M_{i\cdot}= \mu_{\ell_i}^\T$ and $H_{i\cdot}= \epsilon_i^\T$.  Let $\ell=(\ell_i) \in \mathbb{R}^N$ be  the vector of cluster labels.   Let $\nmin = \min_{1\leq j\leq k} n_j$, {$\sigmax = \max_i \norm{\eta_i}_{\psi_2}$}, $\mu_{\max} = \max_{1\leq j\leq k} \norm{\mu_j}_2,$ $\mu_{\text{diff}} = \min_{j\ne m} \norm{\mu_j-\mu_m}_2$,
so that $\mu_{\text{diff}}$ measures the between-class separation. We then define
\[
\snr = \frac{\mudiff^2}{\sigmax^2},\qquad\gamma=\frac{d}{N},
\qquad \zeta = \frac{N}{n_{\min}}, \qquad \xi = \frac{\mumax}{\mudiff}\, , 
\]
{where $\snr$ measures the signal-to-noise ratio}, $\gamma$ measures the ratio between the dimension and the sample size, $\zeta$ measures the {unbalance} of cluster  sizes, and $\xi$ measures the geometric elongation of cluster centers. 

\subsection{The CMDS Embedding}
\label{sec:CMDSembedding}

Applying the double centering procedure as described in the introduction, we obtain the matrix $B$ {from the pairwise distances of possibly unobserved $X$}. For a prespecified embedding dimensionality $r$, the CMDS embedding {of $X$}  is then given by the matrix $Y = \widetilde{V}_r \widetilde{\Lambda}_r^{1/2}$, where the diagonal matrix $\widetilde{\Lambda}_r$ contains the top $r$ eigenvalues of $B$ and the columns of $\widetilde{V}_r$ are the corresponding $r$ eigenvectors.  {To quantify the  embedding quality of the CMDS embedded samples, we} can view  $B=(JX)(JX)^\T$ as a perturbation of {the ideal} $MM^{\T}$. {We emphasize that although the probabilistic  model considered is simple, the error matrix $MM^\T-(JX)(JX)^\T$ is more complex, as it lacks independence in its entries, rows, and columns.}
{Let $V_r\in \mathbb{R}^{N \times r}$ denote the matrix containing the top $r$ eigenvectors of $MM^\T$ and $\Lambda_r = \in \mathbb{R}^{r \times r}$ the diagonal matrix with the corresponding eigenvalues in descending order.}  Write $\rho = \lambda_1/\lambda_r$  as  the ratio between the largest and smallest eigenvalues of the rank $r$ embedding.

\subsection{Exact Cluster Recovery}\label{sec:ExactRecovery}

{To investigate when exact cluster recovery is possible from the noisy CMDS embedding, i.e. when we can}  recover the unobserved cluster vector $\ell$ up to a permutation,  we follow \cite{abbe2018community} and define the  permutation-invariant agreement function as
\begin{align}\label{eq:agreement}
A(u, v) = \max_{\pi \in \mathcal{O}_k} \frac{1}{N} \sum_{i = 1}^N I \big\{u_i = \pi(v_i) \big\},
\end{align}
where $u, v\in [k]^N$ are two vectors of cluster labels,  and $\mathcal{O}_k$ is the set of permutation operators on $[k]$.
We also define the maximal within cluster and minimal between cluster distances in the embedding as
\begin{align*}
d_{\text{in}}(\widetilde{V}_r\widetilde{\Lambda}_r^{1/2}, \ell) &= \max_{i,j, \ell_i=\ell_j} \big\|{(\widetilde{V}_r\widetilde{\Lambda}^{1/2}_r)_{i\cdot} - (\widetilde{V}_r\widetilde{\Lambda}^{1/2}_r)_{j\cdot}}\big\|_2\, , \\
d_{\text{btw}}(\widetilde{V}_r\widetilde{\Lambda}_r^{1/2}, \ell) &= \min_{i,j, \ell_i\ne\ell_j} \big\|(\widetilde{V}_r\widetilde{\Lambda}_r^{1/2})_{i\cdot}-(\widetilde{V}_r\widetilde{\Lambda}^{1/2}_r)_{j\cdot}\big\|_2\, .
\end{align*}
Following \cite{vu2018simple}, we say that the embedding $\widetilde{V}_r\widetilde{\Lambda}_r^{1/2}$ is a {perfect geometric representation} of the labels $\ell$ if $d_{\text{btw}}(\widetilde{V}_r\widetilde{\Lambda}_r^{1/2}, \ell) > 2 d_{\text{in}}(\widetilde{V}_r\widetilde{\Lambda}_r^{1/2}, \ell)$.  
When the number of clusters, $k$, is given, the perfect geometric representation condition is sufficient to guarantee the exact recovery of cluster labels by several common clustering algorithms such as {hierarchical clustering} and $k$-means. {Namely, those methods can} produce an estimated cluster label vector $\hat{\ell} \in [k]^N$ with $A(\hat{\ell}, \ell) = 1$.  
We verify this for $k$-means and {various hierarchical clustering algorithms including} single linkage, {complete linkage, average linkage, and minimum energy}  in Appendix \ref{app:clustering}. 
Since multiple algorithms can correctly recover the labels from a perfect geometric representation, in the remainder of the paper we say that the CMDS embedding  {exactly recovers} the labels when it is a perfect geometric representation of the true labels. 

We thus investigate the sampling regime when this embedding produces a perfect geometric representation with high probability. {Since the CMDS embedding of $M$ simply consists of $k$ distinct points, this occurs when the perturbation of the CMDS embedding of $X$ is small or when the clusters are well separated.} We will derive conditions on the {SNR} {which ensure the latter and thus guarantee the CMDS and clustering procedure achieves} exact recovery with high probability. To elucidate our results, we consider the following three regimes characterized by the relationship between $N$ and $d$: (i). The low dimensional regime assumes $d = O(1)$;   (ii). The moderate dimensional regime assumes $\Omega(1) \leq d \leq O\left([N \log N]^2\right)$; (iii). The high dimensional regime assumes $\Omega\left([N\log N]^2\right) \leq d$.

\section{Main Results}\label{sec:MainResults}

{Treating the multidimensional scaling matrix $B=(JX)(JX)^\T$ as a perturbation of $MM^\T$ as discussed in Section \ref{sec:CMDSembedding}, we establish several theoretical results.} First, the perturbation of the eigenspace is bounded (Theorem \ref{thm:eigenvec_pert}). Second, the perturbation of the CMDS embedding matrix is bounded (Theorem \ref{thm:MDS_embedding}).  Third, a sufficient condition for exactly recovery  is derived (Theorem \ref{thm:MainResult}). Finally, Theorems \ref{thm:MainTheorem_HighD} and \ref{thm:MainTheorem_HighD_ModifiedEigs} extend Theorem \ref{thm:MainResult} to provide a sharp characterization of {exact recovery} in the high-dimensional regime under certain assumptions. To begin with, we need the following  regularity conditions.
\begin{condition}
	\label{cond:balance}
	There exist constants $\tau_1$ and $\tau_2$ such that  
	\[ 
	0<\tau_1< \min(k,\rho,\zeta,\xi) < \max(k,\rho,\zeta,\xi) < \tau_2 < \infty. 
	\] 
\end{condition}
\begin{condition}
	\label{cond:eigenvalue}
	{The embedding dimension satisfies $ r \leq s$, where $s=\textnormal{rank}(MM^\T)$. Furthermore, if $r<s$, then}
	\[ \lambda_{r+1} \leq  \left\{ \frac{\lambda_r}{2\zeta(s-r)} \vee \frac{\mudiff^2 \nmin}{144(s-r)}\right\} .\]
\end{condition}   

For Condition \ref{cond:balance},  $\rho=O(1)$ requires the top $r$ eigenvalues of $MM^{\T}$ to be of the same order, and $\zeta=O(1)$ requires approximately balanced clusters. $\xi = O(1)$ is not necessary for the analysis but allows for concise statements of main results. Condition \ref{cond:eigenvalue} ensures $\lambda_{r+1}$ is sufficiently small relative to $\lambda_r$ and $\mudiff$, which guarantees $r$ dimensions are sufficient  {for data reduction}. 


{Our first theoretical result} bounds the perturbation of the eigenspace associated with the top $r$ eigenvalues of $B$ under {the} max norm. Recall that $\gamma=d/N.$
\begin{theorem}
	\label{thm:eigenvec_pert}
	Under Conditions \ref{cond:balance} and \ref{cond:eigenvalue}, with probability at least $1-O(N^{-1})$ there exists a rotation matrix $R$ such that
	\begin{align*}
	\norm{\widetilde{V}_rR-V_r}_{\max} &\lesssim \frac{\sigmax(\log N)^{1/2}}{\mumax N^{1/2}} +\frac{\sigmax^2}{\mumax^2}\left(\gamma^{1/2}\log N + \frac{\gamma}{N^{1/2}}\right).
	\end{align*}
\end{theorem}
\begin{proof}
The proof  is collected  in Appendix \ref{app:Thm1and2}.
\end{proof}

To further investigate the embedding quality,  we desire a comparison not just of $V_r$ with $\widetilde{V}_r$ but of $V_r\Lambda^{1/2}_r$ with $\widetilde{V}_r\widetilde{\Lambda}^{1/2}_r$, since the rank $r$ CMDS embedding is obtained by rescaling the eigenvectors by the root of the associated eigenvalues.   {Because relative pairwise distances} remain the same under a rotation of the data, it suffices to compare {$V_r\Lambda^{1/2}_r$} with $\widetilde{V}_r\widetilde{\Lambda}^{1/2}_rR$, where 
$R$ is any $r\times r$ rotation matrix. The following theorem summarizes this result. 
\begin{theorem}
	\label{thm:MDS_embedding}
	Assume $\snr \gtrsim 1+\gamma$ and Conditions \ref{cond:balance} and \ref{cond:eigenvalue}. Then  with probability at least $1-O(N^{-1})$ there exists a rotation matrix $R$ such that
	\begin{align*}
	&\norm{V_r\Lambda^{1/2}_r - \widetilde{V}_r\widetilde{\Lambda}^{1/2}_rR}_{\max} \\
	&\qquad\lesssim \left[\sigmax\mumax(1+\sqrt{\gamma})\right]^{1/2}+\sigmax(\sqrt{\log N}+\sqrt{\gamma})+ \frac{\sigmax^2}{\mumax}(\sqrt{d}\log N + \gamma).
	\end{align*}
\end{theorem}
\begin{proof}
The proof  is collected  in Appendix \ref{app:Thm1and2}.
\end{proof}

We emphasize that in practice one does not need to know $R$, since a rotation does not affect any important properties of the embedding.  See for example Theorem \ref{thm:MainResult}, which applies directly to $\widetilde{V}_r\widetilde{\Lambda}^{1/2}_r$, not $\widetilde{V}_r\widetilde{\Lambda}^{1/2}_rR$. Theorem \ref{thm:MDS_embedding} simply demonstrates that the empirical CMDS embedding $\widetilde{V}_r\widetilde{\Lambda}^{1/2}_r$ can be aligned with the perfect model $V_r\Lambda^{1/2}_r$ in such a way that the {elementwise} perturbation is small.

Having bounded the perturbation of the CMDS embedding, we can determine in what regime the exact recovery is possible. A straightforward calculation shows that  the perturbed CMDS embedding $\widetilde{V}_r\widetilde{\Lambda}^{1/2}_r$ is a perfect geometric representation of the community labels when 
\[ 
\mudiff > 12 \sqrt{r}\norm{V_r\Lambda^{1/2}_r - \widetilde{V}_r\widetilde{\Lambda}^{1/2}_rR}_{\max},
\]
as shown in Lemma \ref{lem:RelateER_EmbeddingPert} of Appendix \ref{app:RelateEmbeddingAndRecovery}. Thus, when this condition holds, {the empirical CMDS embedding is a perfect geometric representation of the labels.} The following theorem is obtained by combining this condition with the bound in Theorem \ref{thm:MDS_embedding}, and gives a sufficient condition on the {SNR} to ensure the CMDS and clustering procedure exactly recovers the labels with high probability. 

\begin{theorem}
	\label{thm:MainResult}
	Assume Conditions \ref{cond:balance} and \ref{cond:eigenvalue}. Then with probability at least $1-O(N^{-1})$  the {empirical} rank $r$ CMDS {embedding} $\widetilde{V}_r\widetilde{\Lambda}^{1/2}_r$  is a perfect geometric representation of the labels whenever
	\begin{align} 
	\label{equ:MainCondition}
	\snr &\gtrsim d^{1/2} \log N + \gamma. 
	\end{align}
\end{theorem}
\begin{proof}
The proof  is collected  in Appendix \ref{app:MainResult}.
\end{proof}
Intuitively, $\mudiff^2$ and $\sigmax^2$ quantify the strengths of the signal and the noise respectively. Exact recovery is easier when the signal-to-noise ratio is larger, namely when the classes are well separated relative to the noise. Theorem \ref{thm:MainResult} quantifies the scaling of the signal-to-noise ratio required for exact recovery.  When Inequality \ref{equ:MainCondition} holds, {the empirical CMDS will be a perfect geometric representation of the labels with high probability.}

In the very high-dimensional setting, an alternate analysis using standard Davis-Kahan theorem {\citep{yu2014useful}}  gives sharper results, when extra conditions, including the convex concentration property, hold.   
\begin{condition}
	\label{cond:CCP}
{The noise vectors $\epsilon_i$ have a common covariance matrix $\Sigma$} and satisfy the following convex concentration property with constant $\sigmax$: for every 1-Lipschitz convex function $\phi$,  $\Ex[\phi(\epsilon_i)]<\infty$ and 
	\begin{align*}
	\Prob\left(|\phi(\epsilon_i)-\Ex[\phi(\epsilon_i)]|>t\right) &\leq 2\exp(-t^2/\sigmax^2) \, .
	\end{align*}
\end{condition}

Convex concentration is a common assumption for random vectors in the statistical literature, as the assumption covers many well known models, including a multivariate Gaussian distribution with general covariance and a uniform distribution on the unit sphere \citep{kasiviswanathan2019restricted}, random vectors with bounded independent entries \citep{talagrand1988isoperimetric, talagrand1995concentration}, and random vectors which obey a logarithmic Sobolev inequality \citep{adamczak2005logarithmic}.  

	
\begin{condition}\label{cond:equal_eigs} 
The top $r$ eigenvalues of $MM^\T$ are of {the} same order, that is $\lambda_1 \sim \lambda_r$. More precisely, $\rho-1 \leq C\mudiff/ \mumax$ for some absolute constant $C<1$.
\end{condition}

\begin{theorem}\label{thm:MainTheorem_HighD}
Assume Conditions \ref{cond:balance}--\ref{cond:equal_eigs}. Then with probability at least $1-O(N^{-1})$  the {empirical} rank $r$ CMDS embedding $\widetilde{V}_r\widetilde{\Lambda}^{1/2}_r$ is a perfect geometric representation of the labels whenever
	{
	\begin{align}
	\label{equ:MainConditionEqualEigs}
	\snr &\gtrsim N + d^{1/2}.
	\end{align}
Further{more} if $d\gtrsim N^2$, then $\snr \gtrsim  d^{1/2}$ suffices for a perfect geometric representation.} 
\end{theorem}
\begin{proof}
The proof is collected in Appendix \ref{app:ProofOfMainResultHighD}. 
\end{proof}

{The rate in Theorem \ref{thm:MainResult} is very different from the rate in Theorem \ref{thm:MainTheorem_HighD}.  Letting $P = (JX)(JX)^\T - MM^{\T}$ denote the CMDS error matrix, the former rate is essentially obtained by bounding the perturbation of the CMDS embedding by $\norm{P}_{\infty}/(\mumax N)$, but it can also be bounded by $\norm{P - \mtr(\Sigma)J}_{2}/(\mumax N^{1/2})$ under the assumptions of Theorem \ref{thm:MainTheorem_HighD}. One loses a factor of $N^{1/2}$ in replacing the infinity norm with the spectral norm. However $\norm{P - \mtr(\Sigma)J}_{2}$ scales much more favorably with respect to the dimension than $\norm{P}_{\infty}$, as the centered error matrix can be bounded in terms of $\gamma^{1/2}$ instead of $\gamma$. Altogether, the lower bound on the $\snr$ in Theorem \ref{thm:MainResult} is essentially multiplied by $N^{1/2} \gamma^{-1/2} = N d^{-1/2}$ to obtain the lower bound in Theorem \ref{thm:MainTheorem_HighD}, and this is advantageous in the high-dimensional regime where $d\gtrsim N^2 $. } 

\begin{rmk}
	The convex concentration property (Condition \ref{cond:CCP}) can be replaced with the assumption that each cluster has a Gaussian distribution, that is, $x_i \sim \mathcal{N}( \mu_{\ell_i}, \Sigma)$. Theorem \ref{thm:MainTheorem_HighD} can then be proved, up to logarithmic factors, using a subexponential Matrix Bernstein inequality. See Theorem 6.2 in \cite{tropp2012user} and results in \cite{vershynin2011spectral}. Such an approach allows for Gaussian distributions with a degenerate $\Sigma$. 
\end{rmk}


\begin{rmk}
	{Because the empirical eigenvalues $\widetilde{\lambda}_i$ are inflated by noise accumulations, in the proof of Theorem \ref{thm:MainTheorem_HighD} the noisy CMDS embedding is rescaled before comparing with the ideal embedding. More specifically $\norm{V_r\Lambda^{1/2}_r - \alpha\widetilde{V}_r\widetilde{\Lambda}^{1/2}_rR}_{\max}$ is shown to be smaller than  $\norm{V_r\Lambda^{1/2}_r - \widetilde{V}_r\widetilde{\Lambda}^{1/2}_rR}_{\max}$  for a fixed $\alpha<1$. This rescaling has no effect on the probability of exact recovery.} 
\end{rmk}



By restricting to various regimes of interest, three corollaries follow from Theorems \ref{thm:MainResult} {and \ref{thm:MainTheorem_HighD}}. 
\begin{corollary}[Low Dimensional Regime]
	\label{cor:smalld}
	Under the assumptions of Theorem \ref{thm:MainResult}, if $d = O(1)$, then with probability at least $1-O(N^{-1})$ the {empirical} rank $r$ {CMDS} embedding is a perfect geometric representation of the labels whenever
	\[ \snr \gtrsim \log N. \]	
\end{corollary}

The above condition is natural given that when $d=1$, the maximum norm of a collection of $N$ independent sub-Gaussian random variables with variance approxy $\sigma^2_{\max}$ scales as $\sigmax  (\log N)^{1/2}$  \citep{vershynin2012high}. Thus clearly communities can be exactly recovered only for $\mudiff \gtrsim \sigmax (\log N)^{1/2}$, i.e., $\snr \gtrsim \log N$, and this scaling continues to hold for small $d$.

\begin{corollary}[Moderate Dimensional Regime]
	\label{cor:modd}
	Under the assumptions of Theorem \ref{thm:MainResult}, if  $\Omega(1) \leq d \leq O\left([N \log N]^2\right)$, then with probability at least $1-O(N^{-1})$ the {empirical} rank $r$ {CMDS} embedding {is} a perfect geometric representation of the labels whenever
	\begin{align*} 
	\snr &\gtrsim d^{1/2} \log N. 
	\end{align*}
\end{corollary}

This scaling is better with respect to the dimension than might naively be expected. Consider for example the isotropic case when $\eta_i \sim N(0, \sigmax^2 I)$. Since the data points essentially fall on spheres of radius $\sigmax \sqrt d$ about the means, one might expect that exact recovery requires $\mudiff^2 \gtrsim \sigmax^2 d$ (which gives $\snr \gtrsim d$). However, thanks to the measure concentration phenomenon, the measure of a sphere concentrates rapidly around the equator  in high dimensions, which allows for perfect recovery of the communities even when the spheres of radius $\sigmax d^{1/2}$ about the means intersect. Thus when $N$ is of constant order,  $\snr \gtrsim \sqrt d$ suffices. Recall that  $\gamma={d}/{N}$.


\begin{corollary}[High Dimensional Regime]
	\label{cor:larged}
	Under the assumptions of Theorem \ref{thm:MainResult}, if $d\geq\Omega\left([N\log N]^2\right)$, then with probability at least $1-O(N^{-1})$ the {empirical rank $r$ CMDS} embedding {is} a perfect geometric representation of the labels whenever 
	\begin{align}\label{equ:larged} \snr \gtrsim \gamma \, . \end{align}
	Furthermore, if Conditions \ref{cond:CCP} and \ref{cond:equal_eigs} also hold, then it  suffices that  
	\begin{align}
	\label{equ:larged_rho1}
	\snr \gtrsim  \sqrt{d} .
	\end{align}
\end{corollary}

Inequality \ref{equ:larged_rho1} is a much less restrictive requirement  than Inequality \ref{equ:larged} for exact recovery in high dimensions, but only valid when {$\rho = \lambda_1/\lambda_r\approx 1$}, and the {noise vectors  satisfy the convex concentration property.}  To see this, observe that when $N=O(1)$, Inequality \ref{equ:larged} reduces to $\snr \gtrsim d$ while Inequality \ref{equ:larged_rho1} reduces to $\snr \gtrsim d^{1/2}$.  This is due to the fact that in high dimensions estimating the eigenvalues of $MM^\T$ from the eigenvalues of $(JX)(JX)^\T$ is more problematic than the corresponding eigenspace estimation.  Indeed, since {$\lambda_i\left[\Ex\{(JX)(JX)^\T\}\right] = \lambda_i(J(MM^\T+\mtr(\Sigma)I)J^\T)$}, the  {CMDS} eigenvalues are consistently over-estimated by the order of $\mtr(\Sigma)$. However if $\mtr(\Sigma)$ is known or can be accurately approximated, a simple modification of the CMDS procedure makes Inequality \ref{equ:larged_rho1} sufficient for general $\rho$. Condition \ref{cond:equal_eigs} 
can then be removed from Theorem \ref{thm:MainTheorem_HighD}, and one obtains the following theorem. The proof is a simple modification of the proof of Theorem \ref{thm:MainTheorem_HighD} and is omitted for brevity. 

\begin{theorem}
	\label{thm:MainTheorem_HighD_ModifiedEigs}
	Assume Conditions \ref{cond:balance}--\ref{cond:CCP}. Consider  the modified multidimensional scaling embedding  $\hat{\Lambda}_r^{1/2}\widetilde{V}_r$ instead of $\widetilde{\Lambda}^{1/2}_r\widetilde{V}_r$, where $\hat{\Lambda}_r= \widetilde{\Lambda}_r - \mtr(\Sigma)I_r$. Then with probability at least $1-O(N^{-1})$ the modified  rank $r$ CMDS embedding $\hat{\Lambda}_r^{1/2}\widetilde{V}_r$ is a perfect geometric representation of the labels whenever
	\begin{align*}
	\snr &\gtrsim N + d^{1/2}.
	\end{align*}
\end{theorem}

{Theorem \ref{thm:MainTheorem_HighD_ModifiedEigs} is unsatisfactory because $\Sigma$ is generally unkown. However it does illustrate that a sharper rate can be obtained in high dimensions by unbiasing the eigenvalues.} 

{An important question in practice is how to choose $r$. As illustrated in the last plot of Figure \ref{fig:ExampleMDSEmbedding}, in high dimensions it is critical to strongly reduce the dimension before clustering. There are multiple recommendations in the literature. \cite{cattell1966scree} suggests choosing $r$ as the index corresponding to the ``elbow" in a sorted eigenvalue (scree) plot. \cite{cox2000multidimensional} suggest choosing $r$ by considering $\sum_{i=1}^r \widetilde{\lambda}_i /  \sum_{i=1}^{N-1} \widetilde{\lambda}_i$, which is the proportion of the data variance characterized by a rank $r$ CMDS embedding (this proportion should be large, while $r$ should be small). When the noise is $N(0,\sigma^2I_d)$, \cite{peterfreund2018multidimensional} suggest choosing $r$ as the largest index such that $\widetilde{\lambda}_r > \sigma^2(1 + \sqrt{(N-1)/d})^2$. We suggest choosing $r$ as the index which optimizes the eigenratio, a hueristic which is shown in \cite{lam2012factor} to have several desirable properties. Specifically,
	\begin{equation}
	\label{equ:eigenratio}
	\hat{r} := \text{argmax}_{1 \leq i \leq R} \ \widetilde{\lambda}_{i} / \widetilde{\lambda}_{i+1},
	\end{equation}
	where $R$ is chosen so that only eigenratios computed from eigenvalues which are sufficiently bounded away from zero are considered (for example, choose $R$ as the largest integer such that $\widetilde{\lambda}_{R+1} > 10^{-8}$). We observe that Condition \ref{cond:eigenvalue} has the greatest chance of being satisfied when $r$ is chosen according to this hueristic. We also note that for Condition \ref{cond:eigenvalue} to be satisfied, $r$ can be at most $k-1$.}

{\section{Relationship with Existing Results for Exact Recovery}}

Exact recovery using the CMDS and clustering procedure is related to the stochastic block model, which is extensively studied in the literature. In a stochastic block model, there are $N$ points drawn from $k$ communities, and two points in communities {$q$} and $m$ are connected with probability $P_{q,m}$ \citep{abbe2018community}. The resulting graph {corresponds} to an $N$ by $N$ adjacency matrix $A$, where $A_{i,j} \in \{0,1\}$ {with $\Ex(A_{i,j}) = P_{\ell_i,\ell_j}$ and $\ell$ the label vector with entries in $[k]$}. The matrix $A$ can be viewed as a similarity matrix analagous to the matrix $B$ in multidimensional scaling. However, in contrast to stochastic block models, where the upper triangular entries of $A$ are independent, the corresponding entries of $B$ are dependent because, for example, both $B_{i,j}$ and $B_{i, k}$ depend on {$x_i$}. The {$i$th} row of $A$ can be treated as {an} $N$ dimensional random vector drawn from a probability distribution on $\{0,1\}^N$ with mean {$\mu_{\ell_i} =(P_{\ell_i, \ell_j})_{j \in [N]}$}. With  $\mudiff$ defined as before, and  $\sigmax^2 = \max_{q,m}P_{q,m}(1-P_{q,m})$, \cite{mcsherry2001spectral} showed that the communities can be exactly recovered with high probability when $\snr \gtrsim k(\zeta+\log N)$. See also \cite{bickel2009nonparametric, rohe2011spectral}. Recently, \cite{abbe2016exact} showed that for two balanced communities, this characterization of exact recovery is sharp, i.e., exact recovery fails for $\snr \lesssim \log N$. In the low-dimensional regime and under some mild conditions, we obtain the same lower bound for the signal-to-noise ratio. Although exact recovery is a strong condition and in practice one may only care about recovering a high percentage of correctly labeled points, analysis of exact recovery is closely related to the evaluation of various partitioning algorithms in the stochastic block model literature \citep{bui1987graph, dyer1989solution, boppana1987eigenvalues, snijders1997estimation, bickel2009nonparametric}. As noted in \cite{condon2001algorithms},``Such analyses are useful in that they can provide insight on when and why certain algorithmic approaches are likely to be effective." Similarly, the analysis of the current article provides insight into when CMDS can be effective, {and how the ambient dimension impacts embedding quality.}

Our model is also closely related to the label recovery of a mixture of Gaussian distributions. \cite{achlioptas2005spectral} used spectral projection to show that when $N\gg \zeta k(d+\log k)$, the labels for general Gaussians are exactly recovered for $\snr \gtrsim \zeta + k\log(kN)+k^2$. Building upon the approach in \cite{kumar2010clustering}, \cite{awasthi2012improved} propose a spectral algorithm which improves the dependence on $k$ and correctly classifies a fixed percentage of points when  $\snr \gtrsim 1 \vee \gamma$.
When the Gaussians are spherical and $N$ is sufficiently large, the spectral algorithm proposed by \cite{vempala2004spectral} exactly recovers the distribution labels when $\snr \gtrsim \sqrt{k\log N}$. This condition is less restrictive by a factor of $\sqrt{\log N}$ than our result for the low $d$ regime, but the method requires access to the data coordinates as well as isotropic covariance and is thus not directly comparable.  All of the above methods for recovering a mixture of Gaussians require a singular value decomposition of the data matrix, which is generally not available in the CMDS context. \\

\section{Numerical Studies}\label{sec:SimulationResults}

The CMDS and clustering procedure is applied to synthetic data in Section \ref{subsec:synthetic_data} and to data from real applications in Section \ref{subsec:realworlddata}.

\subsection{Simulations for Synthetic Data}
\label{subsec:synthetic_data}

This section reports simulation results for synthetic data which {illustrate the sharpness of} the scaling conditions in Corollaries \ref{cor:smalld}--\ref{cor:larged}. The simulations consider Gaussian distributions with both isotropic and more general covariance matrices, and for simplicity, use balanced clusters ($\zeta=k$) and $r=\text{rank}(MM^\T)$. Thus for all simulations, Conditions 1--3 are satisfied, except $\rho \gg 1$ in Simulation 2d as described below.

\begin{table}	
	{	\def~{\hphantom{0}}
		\def\arraystretch{1.25}
		\begin{center}
			\begin{small}
				\begin{tabular}{ccccccccc}
					Sim. & $k$ & $d/N$ & $r$ & $J$ & $\Sigma$ & $\{\mu_i\}_{i=1}^k$ & $\rho$ & Plot \\[.1cm]
					1a & 4 & $d=2$ & 2 & 20 & $\sigmax^2I_d$ & $\pm (10^{-7}, 0), \pm (0,10^{-7})$ & 1 & Fig. \ref{fig:fixedD_k4} \\ 
					1b & 4 & $d=10$ & 3 & 50 & Toeplitz & $\mu_i(i)=10^{-7}; \mu_i(j)=0$ for $j\ne i$ & 1 & Fig. \ref{fig:fixedD_Toeplitz} \\
					1c & 4 & $d=20$ & 3 & 50 & KNN & $\mu_i(i)=10^{-7}; \mu_i(j)=0$ for $j\ne i$ & 1 & Fig. \ref{fig:fixedD_KNN} \\ 
					2a & 2 & $N=200$ & 1 & 20 & $\sigmax^2I_d$ & $\mu_i(i)=1; \mu_i(j)=0$ for $j\ne i$ & 1 & Fig. \ref{fig:fixedN_k2} \\
					2b & 5 & $N=100$ & 4 & 20 & $\sigmax^2I_d$ & $\mu_i(i)=0.5; \mu_i(j)=0$ for $j\ne i$ & 1 & Fig. \ref{fig:fixedN_k5_rho1} \\
					2c & 5 & $N=100$ & 4 & 20 & Toeplitz & $\mu_i(i)=0.5; \mu_i(j)=0$ for $j\ne i$ & 1 & Fig. \ref{fig:fixedN_k5_Toeplitz} \\    
					2d & 5 & $N=100$ & 4 & 20 & KNN & $\mu_i(i)=0.5; \mu_i(j)=0$ for $j\ne i$ & 1 & Fig. \ref{fig:fixedN_k5_KNN} \\ 
					2e & 3 & $N=60$ & 2 & 20 & $\sigmax^2I_d$ & $(0,0), (0.4,0.6), (1,1)$ & 75.19 & Fig. \ref{fig:fixedN_k3_rhoLarge} \\
					2f & 5 & $N=100$ & 4 & 20 & $\sigmax^2I_d$ & $(0,0,0,0,\vec{0})$, $\pm (0.49,0.51,0,0, \vec{0})$, & 10,000 & Fig. \ref{fig:fixedN_k5_rhoLarge} \\ 
					& & & & & & $\pm (0,0,0.49,0.51, \vec{0})$, where $\vec{0}\in \mathbb{R}^{d-4}$ & & \\ 
				\end{tabular}
			\end{small}
		\end{center}
	}
	\caption{Simulation settings}
	\label{tab:SimSettings}	
\end{table}	

Table \ref{tab:SimSettings} records the parameters, means, and covariance matrices used for all simulations. Covariance matrices considered include isotropic, Toeplitz, and $K$ nearest neighbor (KNN).  The Toeplitz covariance matrix is defined by $\Sigma_{ij} = \sigmax^2 0.7^{|i-j|}$. The $K$ nearest neighbor covariance matrix is defined as follows: $d$ points $z_1,\ldots,z_d$ are randomly sampling from a two-dimensional cube $[0,c]^2$, and $\Sigma_{ij}=\Sigma_{ji}=\sigmax^2\norm{z_i-z_j}_2$ if $z_i$ is one of $z_j's$ $K$ nearest neighbors, with $\Sigma_{ii}=\sigmax^2$. Simulations 1c, 2d use $K=4,10$ and $c=1, 0.5$ in this construction.

Simulations 1a--1c target the small $d$ regime. The probability of exact recovery using the CMDS and clustering procedure is estimated for fixed $d$, $\mudiff$ on a grid of $\sigmax$, $N$ values. For each grid point, $J$ independent realizations are created, and we record the percentage of realizations for which single linkage clustering on the CMDS embedding perfectly recovers the labels. 

The empirical probability estimates for this set of simulations are plotted in Figure \ref{fig:ExactRecoveryFixedD}; the horizontal axis shows $\log \log N$ and the vertical axis $\log \snr$.  Dark blue indicates exact recovery was achieved by 0\% of the simulations; yellow indicates exact recovery was achieved by 100\% of the simulations.  Since the boundary between exact recovery and failure of exact recovery is linear with a slope of 1 (after an initial plateau in some of the simulations), the boundary is $\log \snr = \log \log N+C$ for some constant $C$, which is equivalent to $\snr = e^{C}\log N$. Since we fix a small $d$, the condition in Corollary \ref{cor:smalld} guaranteeing exact recovery reduces to $\snr \gtrsim \log N$. The simulation results thus confirm that this lower bound for the signal-to-noise ratio is sharp.

\begin{figure}[tb]
	\captionsetup[subfigure]{justification=centering}
	\centering
	\begin{subfigure}{.31\textwidth}
		\includegraphics[width=.95\textwidth]{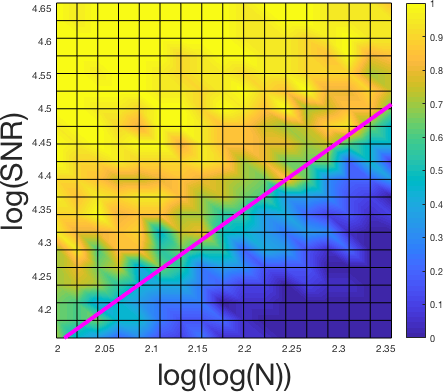}	
		\caption{Isotropic}
		\label{fig:fixedD_k4}
	\end{subfigure}	
	\begin{subfigure}{.31\textwidth}	
		\includegraphics[width=.95\textwidth]{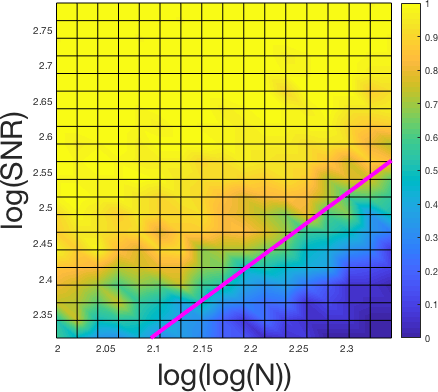}
		\caption{Toeplitz}
		\label{fig:fixedD_Toeplitz}
	\end{subfigure}
	\begin{subfigure}{.31\textwidth}
		\includegraphics[width=.95\textwidth]{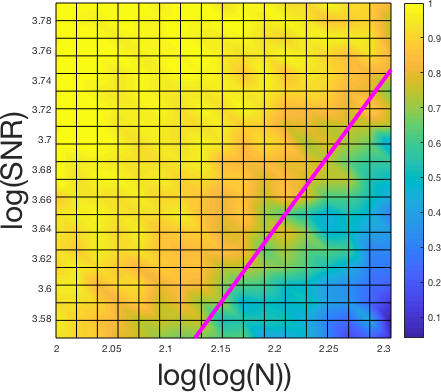}
		\caption{$K$ nearest neighbor}
		\label{fig:fixedD_KNN}
	\end{subfigure}	 
	\caption{
		Pink lines characterize the boundary and are defined by $\log \snr=\log \log N+C$, where $C=2.15, 0.22, 1.44$ for Figures \ref{fig:fixedD_k4}, \ref{fig:fixedD_Toeplitz}, \ref{fig:fixedD_KNN} respectively.}	
	\label{fig:ExactRecoveryFixedD}
\end{figure}

\begin{figure}[tb]
	\centering
	\captionsetup[subfigure]{justification=centering}
	\begin{subfigure}[b]{.31\textwidth}
		\includegraphics[width=.95\textwidth]{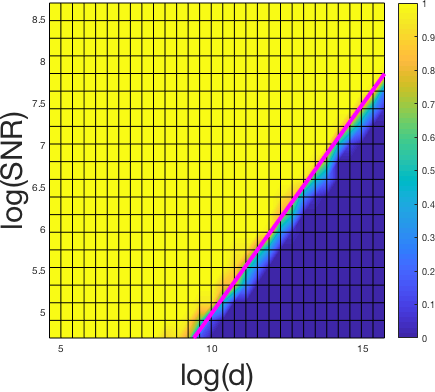}
		\caption{$k=2$, $\rho=1$}
		\label{fig:fixedN_k2}
	\end{subfigure}	
	\begin{subfigure}[b]{.31\textwidth}
		\includegraphics[width=.95\textwidth]{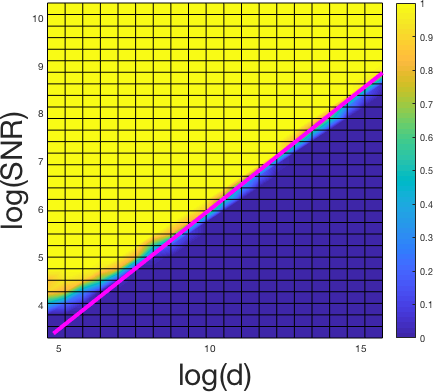}
		\caption{$k=5$, $\rho=1$}
		\label{fig:fixedN_k5_rho1}
	\end{subfigure}	
	\begin{subfigure}[b]{.31\textwidth}
		\includegraphics[width=.95\textwidth]{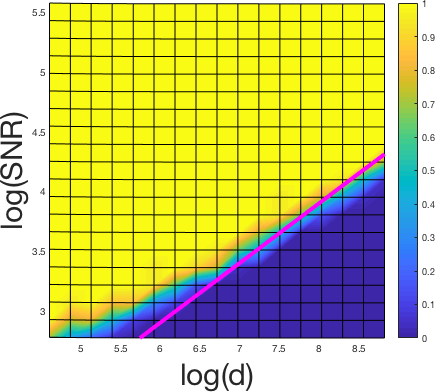}
		\caption{$k=5$, $\rho=1$}
		\label{fig:fixedN_k5_Toeplitz}
	\end{subfigure}	
	\begin{subfigure}[b]{.31\textwidth}
		\includegraphics[width=.95\textwidth]{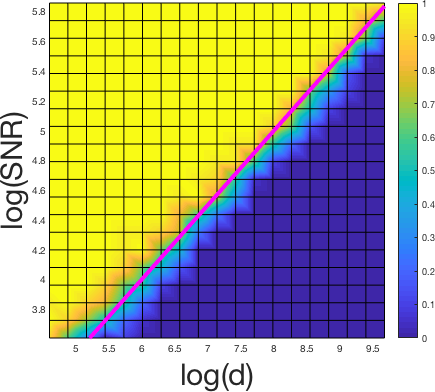}
		\caption{$k=5$, $\rho=1$}
		\label{fig:fixedN_k5_KNN}
	\end{subfigure}		
	\begin{subfigure}[b]{.31\textwidth}
		\includegraphics[width=.95\textwidth]{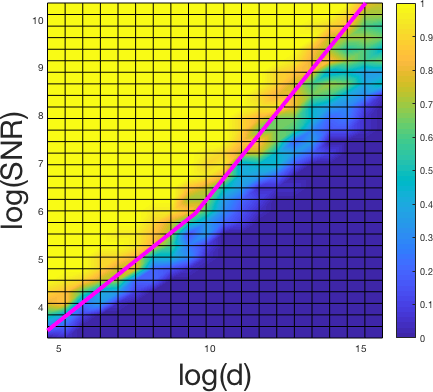}
		\caption{$k=3$, $\rho=75.188$} 
		\label{fig:fixedN_k3_rhoLarge}
	\end{subfigure}	
	\begin{subfigure}[b]{.31\textwidth}
		\includegraphics[width=.95\textwidth]{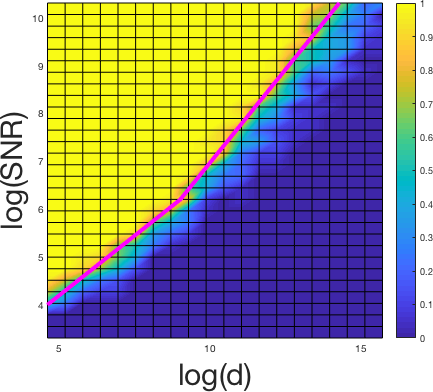}
		\caption{$k=5$, $\rho=10,000$}
		\label{fig:fixedN_k5_rhoLarge}
	\end{subfigure}	
	\caption{
		Pink lines characterize the boundary and are defined by $\log \snr=0.5\log d+C$ where $C=0,1,-0.1,1$ for Figures \ref{fig:fixedN_k2}, \ref{fig:fixedN_k5_rho1}, \ref{fig:fixedN_k5_Toeplitz}, \ref{fig:fixedN_k5_KNN} respectively. For Figures \ref{fig:fixedN_k3_rhoLarge}, \ref{fig:fixedN_k5_rhoLarge}, the left line is $\log \snr=0.5\log d+C$ for $C=1.2,1.7$ and the right line is $\log \snr=0.78\log d+C$ for $C=-1.46,-0.82$.}
	\label{fig:ExactRecoveryFixedN}
\end{figure}

Simulations 2a--2f target the moderate and {high} $d$ regimes. We estimate the probability of exact recovery for fixed $N$, $\mudiff$ on a grid of $\sigmax$, $d$ values in the same manner as for the first set of simulations. Except for Simulations 2c and 2d, $d$ ranges from $100$ to $6,553,600$, so choice of $N$ ensures coverage of both the moderate and  {high} $d$ regimes. For Simulations 2c and 2d, computational resources limited our investigation to the moderate $d$ regime.

The empirical probability estimates for the second set of simulations are plotted in Figure \ref{fig:ExactRecoveryFixedN}; the horizontal axis shows $\log d$ and the vertical axis $\log \snr$.
There is a qualitative similarity between Figures \ref{fig:fixedN_k2}--\ref{fig:fixedN_k5_KNN} (where $\rho=1$) and Figures \ref{fig:fixedN_k3_rhoLarge}--\ref{fig:fixedN_k5_rhoLarge} (where $\rho > 1$). When $\rho=1$ the boundary between exact recovery and failure of exact recovery is linear with a slope of $0.5$. The boundary is thus characterized by $\log(\snr)=0.5\log d +C$ for some constant $C$, which is equivalent to  $\snr=e^C d^{1/2}$. Since $N$ is small and fixed, the condition in Corollary \ref{cor:modd} guaranteeing exact recovery in the moderate $d$ regime reduces to $\snr \gtrsim d^{1/2}$, and since $\rho=1$, the condition in Corollary \ref{cor:larged} guaranteeing exact recovery in the large $d$ regime also reduces to $\snr \gtrsim d^{1/2}$. Thus when $\rho=1$, simulation results confirm that the lower bounds on the signal-to-noise ratio given by Corollaries \ref{cor:modd} and \ref{cor:larged} are sharp. 

When $\rho \ne 1$  (Figures \ref{fig:fixedN_k3_rhoLarge} and \ref{fig:fixedN_k5_rhoLarge}), a phase transition occurs, as the boundary initially characterized by a line with slope $0.5$ (moderate $d$ regime) steepens to a line with slope 0.78 as the dimension is increased (high $d$ regime). The authors hypothesize that the boundary converges to a line with slope 1 as $d\rightarrow \infty$, but computational abilities limited $d\leq6,553,600$. These results confirm that even when $\rho \ne 1$, the scaling condition $\snr \gtrsim \sqrt d$ given by Corollary \ref{cor:modd} for the moderate $d$ regime is sharp. In addition, the steepening of the boundary line suggests that when $\rho\ne 1$, the scaling condition $\snr \gtrsim d$ given by Corollary \ref{cor:larged} for the large $d$ regime may also be sharp. 

\begin{rmk}
	Additional simulation results confirm that if the eigenvalues are unbiased using the procedure described in Theorem \ref{thm:MainTheorem_HighD_ModifiedEigs}, the steepening of the boundary line does not occur for general $\rho$, and the boundary continues to be characterized by $\snr=C d^{1/2}$ in the high dimensional regime. 
\end{rmk}

\subsection{Applications}
\label{subsec:realworlddata}

This section illustrates the usefulness of CMDS for clustering on cancer gene-expression and natural language datasets. We applied CMDS to the Pan-Cancer Atlas gene expression data downloaded from the University of California at Santa Cruz Xena Platform \citep{goldman2019}. The Pan-Cancer Atlas, the output of The Cancer Genome Atlas project, consists of genomic, molecular, and clinical data of  11,060 tumors across 33  cancer types.  For each patient, $d=20,531$ gene expressions with batch correction were recorded. {We constructed two datasets from the Pan-Cancer Atlas: Pan-Cancer1 consisting of breast invasive carcinoma ($n_1 = 1,218$), kidney renal clear cell carcinoma ($n_2=606$), and thyroid carcinoma ($n_4=572$); and Pan-Cancer2 consisting of kidney chromophobe ($n_1=91$), kidney renal clear cell carcinoma ($n_2=606$), and kidney renal papillary cell carcinoma ($n_3=323$). Pan-Cancer1 contains three very different cancer types (breast/kidney/thyroid), while Pan-Cancer2 contained three more similar types of cancer (all kidney). For each data set, a pairwise distance matrix was constructed from the Euclidean distance between the patients' gene expressions. The data sets were then embedded into $r$ dimensions via CMDS as described in Section \ref{sec:ModelFormulation}, and $k$-means clustering was applied to cluster the data into $k=3$ clusters. The embedding dimension $r$ was chosen as the index maximizing the eigenratio as described in Section \ref{sec:MainResults} (see Equation \ref{equ:eigenratio}). Since the data is labeled, we can explicitly compute $M$ and $H$, and thus estimate the $\snr$ by $\widehat{\snr} := \mudiff^2/\hat{\sigma}^2_{\max}$, where $\hat{\sigma}^2_{\max} := \norm{H^{\T}H/N}_2$. For Pan-Cancer1, $\widehat{\snr}=16.64$, and for Pan-Cancer2, $\widehat{\snr}=4.56$, which indicates that Pan-Cancer2 is a more difficult data set to cluster. Results are summarized in Table \ref{tab:cancerdata}, and the CMDS embeddings are shown in Figures \ref{fig:PANCAN1} and \ref{fig:PANCAN2}. We obtained high accuracy in clustering cancer type on Pan-Cancer1 (99.12\%), and a lower accuracy (83.24\%) on Pan-Cancer2. In addition, Figure \ref{fig:PANCAN1} shows that the Pan-Cancer1 clusters are well separated in the CMDS embedding, while Figure \ref{fig:PANCAN2} shows the Pan-Cancer2 clusters are more mixed.} We emphasize these results were obtained only through pair-wise distances between samples without the need to access full data. 

%


\begin{table}
	\begin{center}
		{
		\begin{tabular}{c c c c c}
	Data set & Sample size ($N$)  & Embedding rank ($r$) & Estimated $\snr$ & Accuracy \\ 
	Pan-Cancer1 & $2,396$ & 2 & 16.64 &99.12\% \\
	Pan-Cancer2  & $1,020$ & 3 & 4.56 & 83.24\% 
\end{tabular}
		}
 		\caption{Classification accuracy for Pan-Cancer}	    
		\label{tab:cancerdata}
	\end{center}	
\end{table}

We also applied CMDS to the natural language datasets available from \url{http://www-alg.ist.hokudai.ac.jp/datasets.html}. 
The data points considered are the names of famous individuals, taken from four possible classes: classical composers, artists, authors, and mathematicians; the data points thus do not have geometric coordinates. The pair-wise dissimilarity between data points is then defined using Google distance, a metric suggested by \cite{cilibrasi2006automatic} for sets of natural language terms. The Google distance between two terms is computed from the relative frequency count of the number of web pages containing both terms returned by an automated web search \citep{poland2006clustering}. We analyzed two such datasets: people3 ($k=3,N=75$), a natural language data set containing the names of 25 classical composers, 25 artists, and 25 authors; and people4 ($k=4,N=100$), the people3 data set with the names of 25 mathematicians added.

{Once again we selected the embedding dimension $r$ using Equation \ref{equ:eigenratio}, and ran $k$-means clustering on the resulting CMDS. Since data coordinates are not available, we must estimate the $\snr$ from the full-dimensional CMDS embedding.  Since the data is not Euclidean, the pairwise distance matrices are no longer postitive semi-definite; negative eigenvalues were thus discarded when computing the full-dimensional embedding, so that the pairwise distance matrix is projected onto the closest positive semi-definite matrix. The resulting clustering accuracies are in Table \ref{tab:peopledata}, and Figures \ref{fig:people3} and \ref{fig:people4} show the first two coordinates of the embedding for People3 and People4. The estimated $\snr$ is quite similar for the two data sets, as is the clustering accuracy.}
For people4, the third embedding coordinate is needed to separate the clusters but is not displayed in Figure \ref{fig:people4}.

\begin{table}[tbh]
		\begin{center}
		{
		\begin{tabular}{c c c c c}
			Data set & Sample size ($N$)  & Embedding rank ($r$) & Estimated $\snr$ & Accuracy \\ 
			People3  & $75$ & 2 & 11.22 & 97.33\% \\
			People4 & $100$ & 3 & 11.74 & 94.00\% \\ 
		\end{tabular}
		}
		\caption{Classification accuracy for natural language data sets}
		\label{tab:peopledata}
		\end{center}	    
\end{table}

\begin{figure}[tb]
	\captionsetup[subfigure]{justification=centering}
	\centering
	\begin{subfigure}{.225\textwidth}
		\includegraphics[width=.95\textwidth]{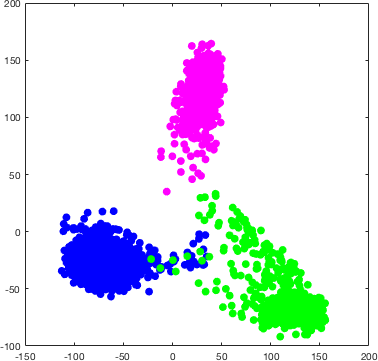}
		\caption{Pan-Cancer1}
		\label{fig:PANCAN1}
	\end{subfigure}	
	\begin{subfigure}{.225\textwidth}
		\includegraphics[width=.95\textwidth]{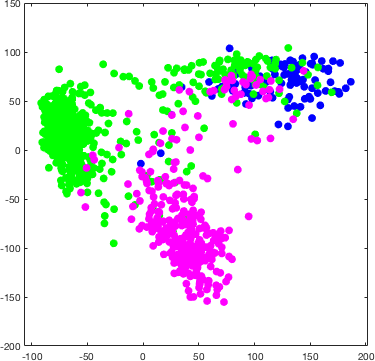}
		\caption{Pan-Cancer2}
		\label{fig:PANCAN2}
	\end{subfigure}	
	\begin{subfigure}{.225\textwidth}
		\includegraphics[width=.95\textwidth]{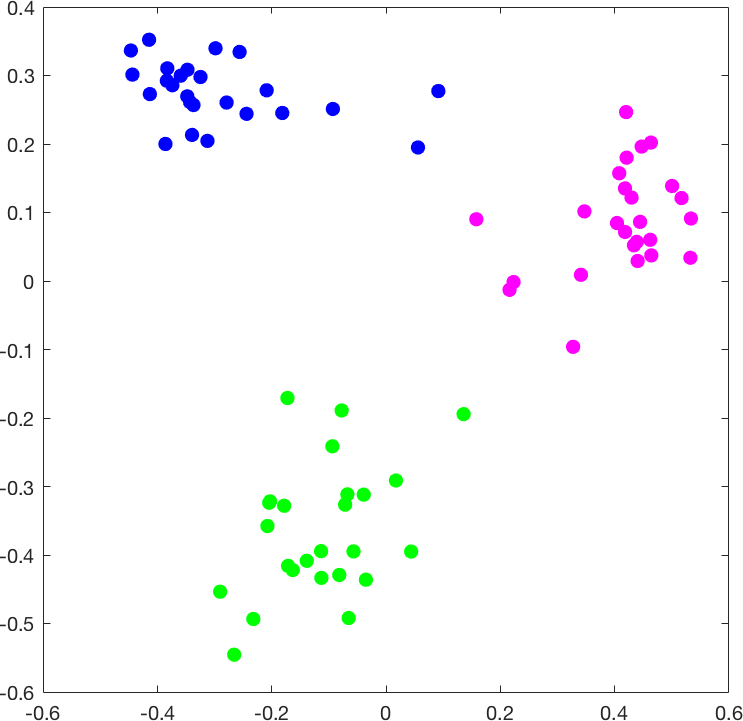}
		\caption{People3}
		\label{fig:people3}
	\end{subfigure}		
	\begin{subfigure}{.225\textwidth}
		\includegraphics[width=.95\textwidth]{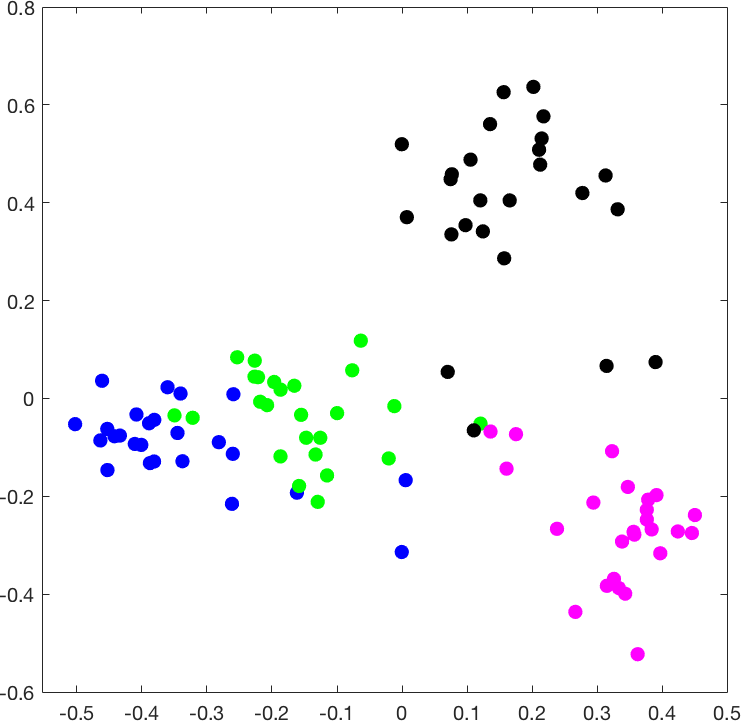}
		\caption{People4}
		\label{fig:people4}
	\end{subfigure}	
	\caption{First two coordinates of the embedding obtained by applying CMDS to Pan-Cancer and people data sets. Each color corresponds to a cancer type for Pan-Cancer and a people group for people data sets.}
	\label{fig:people}
\end{figure}

\section{Discussion}

This paper lays down a theoretical framework for analyzing the embedding quality of CMDS {for noisy data.} Based on this framework, we further derive sharp scaling conditions on the signal-to-noise ratio {which guarantee}  a follow-up clustering procedure  {will} exactly recover the labels  {for} a collection of sub-Gaussian communities. These results generalize scaling conditions for exact recovery in stochastic block models to the more complex case of a network with dependent edges, giving the first strong theoretical guarantees for CMDS in the literature. We make no assumptions on the relationship between $N$ and $d$, and numerical simulations confirm the derived scaling conditions are near-sharp. 
{In addition, our results illustrate 
how CMDS may break down in the presence of high-dimensional noise.}
Many interesting open questions remain for future research, including extending the current framework to {include} {heavy-tailed} noise distributions, more general metrics, and {incorporating} sparse measurement error in the pairwise distance matrix. Our results suggest that the performance of CMDS can be significantly improved by unbiasing the sample eigenvalues in the high dimensional regime. It remains to develop and analyze empirically robust unbiasing procedures.

\appendix 
\renewcommand{\theequation}{S.\arabic{equation}}
\renewcommand{\thetable}{S.\arabic{table}}
\renewcommand{\thefigure}{S.\arabic{figure}}
\renewcommand{\thesection}{S.\arabic{section}}
\renewcommand{\thelemma}{S.\arabic{lemma}}

\vspace{30pt}
\noindent{\bf \LARGE Appendix}

\section{Proofs for Main Theorems}
\label{app:Thm1and2}

This section collects the proofs of the main theorems, that is, Theorems \ref{thm:eigenvec_pert}, \ref{thm:MDS_embedding},  \ref{thm:MainResult}, \ref{thm:MainTheorem_HighD} and \ref{thm:MainTheorem_HighD_ModifiedEigs}. For notational simplicity, we introduce  the error matrix  corresponding to the noisy CMDS embedding 
\begin{align}
\label{equ:error_matrix}
P &:= (JX)(JX)^\T - MM^{\T}, \,
\end{align}
which will be frequently referred to throughout the appendix. 

\subsection{Proofs of Theorems \ref{thm:eigenvec_pert} and \ref{thm:MDS_embedding}}
Before proving Theorems \ref{thm:eigenvec_pert} and \ref{thm:MDS_embedding}, we first provide a lemma on the structure of the eigenvectors of $MM^\T$ by establishing that the eigenvectors  are delocalized. This structure lemma is critical for our proof.
\begin{lem}
	\label{lem:eigvec_delocaliztion}
	Let $v$ be an eigenvector of $MM^{\T}$ corresponding to a non-zero eigenvalue. Then $\norm{v}_{\infty} \leq n_{\min}^{-1/2}$.
\end{lem}
\begin{proof}
	Let $v$ be an eigenvector of $MM^{\T}$ corresponding to a non-zero eigenvalue  $\lambda$.  We first establish that $v$ is constant on the blocks defined by the community labels $\ell$, i.e. $v(i)=v(j)$ whenever $\ell_i=\ell_j$. Ordering the points according to $\ell$, $MM^{\T}$ has the following block constant form
	\begin{align*}
	MM^{\T} &= \begin{bmatrix} \langle \mu_1, \mu_1\rangle 1_{n_1}1_{n_1}^{\T} & \langle \mu_1, \mu_2\rangle 1_{n_1}1_{n_2}^{\T} & \ldots &  \langle \mu_1, \mu_k\rangle 1_{n_1}1_{n_k}^{\T} \\
	\langle \mu_2, \mu_1\rangle 1_{n_2}1_{n_1}^{\T} & \langle \mu_2, \mu_2\rangle 1_{n_2}1_{n_2}^{\T} & \ldots &  \langle \mu_2, \mu_k\rangle 1_{n_2}1_{n_k}^{\T} \\
	 \vdots & & & \vdots \\
	 \langle \mu_k, \mu_1\rangle 1_{n_k}1_{n_1}^{\T} & \langle \mu_k, \mu_2\rangle 1_{n_k}1_{n_2}^{\T} & \ldots &  \langle \mu_k, \mu_k\rangle 1_{n_k}1_{n_k}^{\T}
				\end{bmatrix}, \, 
	\end{align*}
	so that
	\begin{align*}
	(MM^{\T}v)(i) &= \langle \mu_{\ell_i}, \mu_1\rangle \sum_{j=1}^{n_1} v(j) +\langle \mu_{\ell_i}, \mu_2\rangle \sum_{j=n_1+1}^{n_1+n_2} v(j) + \ldots + \langle \mu_{\ell_i}, \mu_k\rangle \sum_{j=N-n_k+1}^{N} v(j) \, .
	\end{align*}
	For any $(i,j)$ such that $\ell_i=\ell_j$, the above calculation indicates that $(MM^{\T}v)(i) =(MM^{\T}v)(j)$. Since $\lambda\ne 0$, 
	\begin{align*}
	\lambda v(i) &= (MM^{\T}v)(i) =(MM^{\T}v)(j)=\lambda v(j) \quad \implies \quad v(i)=v(j)\, .
	\end{align*} 
	Since $v$ is block constant, the entry achieving $\norm{v}_{\infty}$ is repeated at least $n_{\min}$ times, so that
	\begin{align*}
	1 &= \norm{v}_2^2 \geq  n_{\min}\norm{v}_{\infty}^2 \quad \implies \quad \norm{v}_{\infty} \leq n_{\min}^{-1/2}\, .
	\end{align*} 
\end{proof}

\noindent \textbf{Proof of Theorem  \ref{thm:eigenvec_pert}}
	To bound $\norm{\widetilde{V}_rR - V_r}_{\max}$, 
	we apply Theorem \ref{thm:LowRankInfinityEigPert} in Appendix \ref{app:EigPertResults} to obtain entry-wise control on the perturbation of the eigenspace. 
	We first bound $\kappa$ and $\epsilon$. {By Lemma \ref{lem:eigvec_delocaliztion}, we have  $\norm{v_i}_{\infty} \leq n_{\min}^{-1/2}$ for $1\leq i\leq r$,}  and thus
	\begin{align*}
	\kappa &\leq \frac{N}{r}\left(\norm{v_1}^2_{\infty}+\cdots+\norm{v_r}^2_{\infty}\right) \leq \frac{N}{r}\left(\frac{r}{\nmin}\right) =\zeta \, .
	\end{align*}
	Thus we have low coherence since $\zeta=O(1)$ by Condition \ref{cond:balance}. Also since $\norm{v_iv_i^\T}_\infty \leq N\norm{v_iv_i^\T}_{\max}=N\norm{v_i}^2_{\infty} \leq N\nmin^{-1} = \zeta$, we have
	\begin{align*}
	\epsilon &= \norm[\Big]{ \sum_{i=r+1}^s \lambda_iv_iv_i^\T }_{\infty} 
	\leq \lambda_{r+1}(s-r)\max_i\norm{v_iv_i^T}_\infty \leq \lambda_{r+1}(s-r)\zeta \leq \frac{\lambda_r}{2},
	\end{align*}
	since by Condition \ref{cond:eigenvalue} $\lambda_{r+1} \leq \lambda_r/\{2\zeta(s-r)\}$. Recalling $MM^\T$ is positive semi-definite, $|\lambda_r| - \epsilon \geq \lambda_r/2$. Thus by Theorem \ref{thm:LowRankInfinityEigPert} there exists a rotation matrix $R$ such that $\norm{\widetilde{V}_rR-V_r}_{\text{max}} \lesssim r^{5/2}\zeta^2\norm{\error}_{\infty}  / ( \lambda_r N^{1/2} )  \lesssim  \norm{\error}_{\infty}/( \lambda_r N^{1/2})$.  We next verify that $\lambda_i \sim N\mumax^2$ for $1\leq i\leq r$. Since $\lambda_1+\ldots+\lambda_s = \norm{M}_\rF^2$ and $n_{\min} \mumax^2 \leq \norm{M}_\rF^2 \leq N \mumax^2$, one obtains 
	\begin{align*}
	\frac{1}{\zeta(k-1)}N\mumax^2 \leq \frac{1}{s}\nmin\mumax^2 &\leq \lambda_1 \leq N \mumax^2 \, .
	\end{align*} 
	Since $k,\zeta=O(1)$, $\lambda_1 \sim N \mumax^2$ and  $\rho=O(1)$, we have $\lambda_i \sim N \mumax^2$ for $i\leq r$. Thus
	\begin{align*}
	\norm{\widetilde{V}_rR-V_r}_{\max} &\lesssim \frac{\norm{P}_{\infty}}{\lambda_r N^{1/2}} \lesssim \frac{\norm{P}_{\infty}}{\mumax^2 N^{3/2}} \, .
	\end{align*}
	Applying Lemma \ref{lem:RMTInfinityControl} in Appendix \ref{app:RMTLemmas} to bound $\norm{P}_{\infty}$ completes the proof.	
\begin{flushright} \qed \end{flushright}

\noindent \textbf{Proof of Theorem \ref{thm:MDS_embedding}}
	First note that $\snr \gtrsim 1 + \gamma$ and Lemma \ref{lem:RMTSpectralControl} guarantee $\norm{P}_2 \lesssim \lambda_r$ with probability at least $1-O(N^{-1})$. Thus by Lemma \ref{lem:embedding_pert_general} in Appendix \ref{app:EmbedPertResults},  
	\begin{align*}
	\norm{\widetilde{V}_r\widetilde{\Lambda}^{1/2}_rR - V_r\Lambda^{1/2}_r}_{\max} &\lesssim \frac{\norm{\error}_2}{(\lambda_r N)^{1/2}} + \frac{\norm{\error}_2^{1/2}}{N^{1/2}}+ \frac{\norm{\error}_{\infty}}{(\lambda_r N)^{1/2}} \, .
	\end{align*}
	Applying Lemmas \ref{lem:RMTSpectralControl} and \ref{lem:RMTInfinityControl} to bound $\norm{\error}_2, \norm{\error}_{\infty}$ completes the proof.	
\begin{flushright} \qed \end{flushright}

\subsection{Proof of Theorem \ref{thm:MainResult} }
\label{app:MainResult}
{Inequality \ref{equ:MainCondition}, i.e. $\snr\gtrsim  \sqrt{d}\log N +\gamma$,  {guarantees $\mumax^2 \gtrsim \sigmax^2(1+\gamma)$}  since $\xi\sim 1$, which further ensures $\mumax^2 \gtrsim \sigmax\mumax(1+\sqrt{\gamma})$. Thus by Lemma \ref{lem:RMTSpectralControl}, the error matrix $P$ defined in Equation \ref{equ:error_matrix} satisfies $\norm{P}_2/N \lesssim \sigmax\mumax(1+\sqrt{\gamma})+\sigmax^2(1+\gamma) \lesssim \mumax^2 \lesssim \lambda_r/N$ with probability at least $1-O(N^{-1})$, which ensures $\norm{P}_2 \lesssim \lambda_r$.} Thus by  Lemmas \ref{lem:embedding_pert_general} and \ref{lem:RelateER_EmbeddingPert}, it suffices to show
\begin{align}
\label{equ:mudiff_condA}
\mudiff &\gtrsim \frac{\norm{\error}_2}{(\lambda_r N)^{1/2}} + \frac{\norm{\error}_2^{1/2}}{N^{1/2}} +\frac{\norm{\error}_{\infty}}{(\lambda_r N)^{1/2}} \, .
\end{align}
Applying Lemmas \ref{lem:RMTSpectralControl} and \ref{lem:RMTInfinityControl} and utilizing $\lambda_r \gtrsim \mumax^2 N$ {(see the proof of Theorem \ref{thm:eigenvec_pert})}, we obtain
\begin{align*}
\frac{\norm{\error}_2}{(\lambda_r N)^{1/2}} &\lesssim \sigmax(1+\sqrt{\gamma})+\frac{\sigmax^2}{\mumax}(1+\gamma) \\
\frac{\norm{\error}_2^{1/2}}{N^{1/2}} &\lesssim \left[\sigmax\mumax(1+\sqrt{\gamma})\right]^{1/2} + \sigmax(1+\gamma)^{1/2} \\
\frac{\norm{\error}_{\infty}}{(\lambda_r N)^{1/2}}&\lesssim \sigmax\sqrt{\log N} + \frac{\sigmax^2}{\mumax}(\sqrt{d}\log N + \gamma)
\end{align*}
with probability at least $1-O(N^{-1})$. Plugging the above bounds into Inequality \ref{equ:mudiff_condA}, it is sufficient for
\begin{align*}
\mudiff &\gtrsim \left[\sigmax\mumax(1+\sqrt{\gamma})\right]^{1/2}+\sigmax(\sqrt{\log N}+\sqrt{\gamma})+ \frac{\sigmax^2}{\mumax}(\sqrt{d}\log N + \gamma)\, .
\end{align*}
Writing this as three distinct inequalities (one for each term on the right hand side) and solving for $\snr$ yields that exact recovery is obtained when all of the following hold:
\begin{align*}
\snr &\gtrsim \xi^2(1+\gamma) \\
\snr &\gtrsim \log N + \gamma \\
\snr &\gtrsim \frac{1}{\xi}(\sqrt{d}\log N + \gamma)\, ,
\end{align*}
where $\xi = \mumax/\mudiff$.  Since $\xi = O(1)$ by Condition \ref{cond:balance}, the last inequality implies the first two, and the theorem is proved.	  \hfill \qed

\subsection{Proof of Theorem \ref{thm:MainTheorem_HighD}}
\label{app:ProofOfMainResultHighD}

{Since $\xi\sim 1$ and Inequality \ref{equ:MainConditionEqualEigs} guarantees $\sigmax\mumax+\sigmax^2(\sqrt{\gamma}\vee 1) \lesssim \mumax^2$, by Lemma \ref{lem:RMTCenteredSpectralControl}, the error matrix $P$ defined in Equation \ref{equ:error_matrix} satisfies $\norm{P-\mtr(\Sigma)J}_2/N \lesssim \sigmax\mumax+\sigmax^2(\sqrt{\gamma}\vee 1) \lesssim \mumax^2 \lesssim \lambda_r/N$ with probability at least $1-O(N^{-1})$, which ensures $\norm{P-\mtr(\Sigma)J}_2 \lesssim \lambda_r$.} Thus by combining Lemmas \ref{lem:embedding_pert_CommonCovCCP} and \ref{lem:RelateER_EmbeddingPert}, to prove the theorem, it suffices to show
\begin{align}
\label{equ:mudiff_condB}
\mudiff &\gtrsim \frac{\tau_{\rho}}{N^{1/2}} + \frac{ \norm{\error - \mtr(\Sigma)J}_2^{1/2}}{N^{1/2}}+ \frac{\norm{\error - \mtr(\Sigma)J}_2}{\lambda_r^{1/2}},
\end{align}
where 
\[ \tau_\rho =\frac{\lambda_r^{1/2}(\rho-1)}{\big\{\lambda_1/\mtr(\Sigma)+1\big\} } \, . \] 
Applying Lemma \ref{lem:RMTCenteredSpectralControl} and utilizing {$\lambda_r \sim \mumax^2 N$} {(see the proof of Theorem \ref{thm:eigenvec_pert})}, we obtain
\begin{align*}
\frac{ \norm{\error - \mtr(\Sigma)J}_2^{1/2}}{N^{1/2}} &\lesssim \left( \sigmax\mumax \right)^{1/2}  + \sigmax \left(\gamma^{1/4} \vee 1\right) :=  \text{(I)}+\text{(II)}, \\
 \frac{\norm{\error - \mtr(\Sigma)J}_2}{\lambda_r^{1/2}} &\lesssim \sigmax N^{1/2}  + \frac{\sigmax^2}{\mumax}N^{1/2}\left(\gamma^{1/2} \vee 1\right):= \text{(III)}+ \text{(IV)},
\end{align*}
with probability at least $1-O(N^{-1})$. Using  $\lambda_r\sim \mumax^2 N$, $\lambda_1/\mtr(\Sigma)+1\geq1$ and Condition \ref{cond:equal_eigs}, $\mudiff/\mumax \gtrsim \rho-1$  implies that  
\$
\mudiff \gtrsim \lambda_r^{1/2}(\rho-1)/\left[\left\{\lambda_1/\mtr(\Sigma)+1\right\}N^{1/2}\right] = \tau_{\rho}/N^{1/2}.
\$ 
Thus, 
to prove Inequality \ref{equ:mudiff_condB}, we only need to show
\begin{align*}
\mudiff &\gtrsim \text{(I)}+\text{(II)}+\text{(III)}+\text{(IV)}.
\end{align*}
Since $\text{(I)}+\text{(II)}+\text{(III)}+\text{(IV)}\sim \text{(I)}\vee \text{(II)}\vee \text{(III)}\vee\text{(IV)}$, the exact recovery is obtained when all of the following hold:
{\begin{align*}
\snr &\gtrsim \xi^2,\quad
\snr \gtrsim \left(\gamma^{1/2} \vee 1\right),\quad
\snr \gtrsim N,\quad
\snr \gtrsim \frac{1}{\xi} N^{1/2}\left( \gamma^{1/2} \vee 1\right),
\end{align*}}
where $\xi = \mumax/\mudiff$. Since $\xi = O(1)$ by Condition \ref{cond:balance}, the fourth inequality implies the second and the third inequality implies the first. Thus the theorem holds if
{\begin{align*}
\snr &\gtrsim N + d^{1/2}.
\end{align*}}	
This completes the proof.    \hfill \qed

\subsection{Proof of Theorem \ref{thm:MainTheorem_HighD_ModifiedEigs}}
We simply need to establish the following modified version of Lemma \ref{lem:embedding_pert_CommonCovCCP}:
\begin{align*}
\norm{\widetilde{V}_r\hat{\Lambda}^{1/2}_rR - V_r\Lambda^{1/2}_r}_{\max} &\lesssim \frac{ \norm{\error - \mtr(\Sigma)J}_2^{1/2}}{N^{1/2}}+ \frac{\norm{\error - \mtr(\Sigma)J}_2}{\lambda_r^{1/2}} \, .
\end{align*}
The proof of Lemma \ref{lem:embedding_pert_CommonCovCCP} is modified by taking $\alpha=1$ and replacing $\widetilde{\lambda}_i$ with $\hat{\lambda}_i$ (recall $\hat{\lambda}_i = \widetilde{\lambda}_i - \mtr(\Sigma)$). The $\tau_\rho$ term is eliminated.
More precisely, from the argument in Lemma \ref{lem:embedding_pert_general}, we have
\begin{align*}
\norm{\widetilde{V}_r\hat{\Lambda}^{1/2}_rR - V_r\Lambda^{1/2}_r}_{\max}
&\lesssim \underbrace{ \hat{\lambda}_1^{1/2}\norm{\widetilde{V}_rR-V_r}_{\max}}_{:=\Gamma_1}
+\underbrace{\norm{V_r}_{\max}\norm{\hat{\Lambda}^{1/2}_rR-R\Lambda^{1/2}_r}_{\max}}_{:=\Gamma_2}.
\end{align*}
Since 
\begin{align*}
 \hat{\lambda}_1^{1/2}	&\leq \frac{\norm{\error-\mtr(\Sigma)J}_2}{\lambda_1^{1/2}}+\lambda_1^{1/2} \lesssim \lambda_1^{1/2} \, ,
\end{align*}
(this was the previous bound for $\alpha\widetilde{\lambda}_1^{1/2}$), $\Gamma_1$ can be bounded in exactly the same way as in Lemma \ref{lem:embedding_pert_CommonCovCCP}.
In addition, since
\begin{align*}
|\hat{\lambda}_i^{1/2}-\lambda_j^{1/2}| &\leq \frac{|\hat{\lambda}_i-\lambda_i|+|\lambda_i-\lambda_j|}{\lambda_j^{1/2}} \leq   \frac{\norm{\error-\mtr(\Sigma)J}_2}{\lambda_r^{1/2}} + \frac{|\lambda_i-\lambda_j|}{\lambda_r^{1/2}} \, ,
\end{align*}
$\Gamma_2$ can be bounded in the same way as in Lemma \ref{lem:embedding_pert_CommonCovCCP} but omitting the $\tau_\rho$ term (previously $|\alpha\widetilde{\lambda}_i^{1/2}-\lambda_j^{1/2}|$ was bounded by the above plus $\tau_\rho$.) The proof of Theorem \ref{thm:MainTheorem_HighD_ModifiedEigs} is then identical to the proof of Theorem \ref{thm:MainTheorem_HighD}, except that Condition \ref{cond:equal_eigs} is not needed to control the $\tau_\rho$ term.   \hfill \qed

\section{Two Embedding Perturbation Lemmas}
\label{app:EmbedPertResults}

In this Appendix we state and prove two lemmas, which are needed for proving the main theorems. The first, which holds under more general assumptions, is used to prove Theorems {\ref{thm:MDS_embedding} and} \ref{thm:MainResult}; the second 
is used to prove Theorem \ref{thm:MainTheorem_HighD}. Recall that $P = (JX)(JX)^\T - MM^{\T}.$

\begin{lem}
	\label{lem:embedding_pert_general}
	Assume Conditions \ref{cond:balance}, \ref{cond:eigenvalue},  and $\norm{\error}_2 \lesssim \lambda_r$. Then there exists a rotation matrix $R$ such that
	\begin{align*}
	\norm{\widetilde{V}_r\widetilde{\Lambda}_r^{1/2}R - V_r\Lambda_r^{1/2}}_{\max} &\lesssim \frac{\norm{\error}_2}{(\lambda_r N)^{1/2}} + \frac{\norm{\error}_2^{1/2}}{N^{1/2}} +\frac{\norm{\error}_{\infty}}{(\lambda_r N)^{1/2}}  \, .
	\end{align*}
\end{lem}

\begin{proof}
	We upper bound the left hand side as
	\begin{align*}
	&\norm{\widetilde{V}_r\widetilde{\Lambda}^{1/2}_rR - V_r\Lambda^{1/2}_r}_{\max} \\
	&\qquad \leq \ \norm{\widetilde{V}_rRR^{-1}\widetilde{\Lambda}^{1/2}_rR - V_rR^{-1}\widetilde{\Lambda}^{1/2}_rR}_{\max}  + \norm{V_rR^{-1}\widetilde{\Lambda}^{1/2}_rR - V_r\Lambda^{1/2}_r}_{\max} \\
	&\qquad\leq \norm{\widetilde{V}_rR-V_r}_{\max}\norm{R^{-1}\widetilde{\Lambda}^{1/2}R}_{1} + \norm{V_r}_{\max}\norm{R^{-1}\widetilde{\Lambda}^{1/2}_rR-\Lambda^{1/2}_r}_1 \\
	&\qquad \leq r^{1/2}\norm{\widetilde{V}_rR-V_r}_{\max}\norm{R^{-1}\widetilde{\Lambda}^{1/2}_rR}_{2}  +  r^{1/2}\norm{V_r}_{\max}\norm{R^{-1}\widetilde{\Lambda}^{1/2}_rR-\Lambda^{1/2}_r}_2 \\
	&\qquad=(r\widetilde{\lambda}_1)^{1/2}\norm{\widetilde{V}_rR-V_r}_{\max}+r^{1/2}\norm{V_r}_{\max}\norm{\widetilde{\Lambda}^{1/2}_rR-R\Lambda^{1/2}_r}_2 \\
	&\qquad\leq  \underbrace{(r\widetilde{\lambda}_1)^{1/2}\norm{\widetilde{V}_rR-V_r}_{\max}}_{:=\Gamma_1}+\underbrace{r^{3/2}\norm{V_r}_{\max}\norm{\widetilde{\Lambda}^{1/2}_rR-R\Lambda^{1/2}_r}_{\max}}_{:=\Gamma_2}.
	\end{align*}	
	Since $r \leq k-1 = O(1)$ by Condition \ref{cond:balance}, we ignore constants depending on $r$. We first control $\Gamma_1$. Since $|\widetilde{\lambda}_i-\lambda_i|\leq \norm{\error}_2$ by Weyl's Inequality and  $\norm{\error}_2 \lesssim \lambda_r$ by assumption, we have
	\begin{align*}
	\widetilde{\lambda}_1^{1/2} 
	&=\lambda_1^{1/2} +\frac{\widetilde{\lambda}_1-\lambda_1}{\widetilde{\lambda}_1^{1/2}+\lambda_1^{1/2}}
	\leq \lambda_1^{1/2} + \frac{|\widetilde{\lambda}_1-\lambda_1|}{\lambda_1^{1/2}} \leq \lambda_1^{1/2} + \frac{\norm{\error}_2}{\lambda_1^{1/2}} \lesssim \lambda_1^{1/2}\, ,
	\end{align*}	
	so that
	\begin{align*}
	\Gamma_1 &\lesssim (\lambda_1)^{1/2}\norm{\widetilde{V}_rR-V_r}_{\max}\, .
	\end{align*}
	To bound $\Gamma_2$, we partition the eigenvalues $\lambda_1, \ldots, \lambda_r$ into $\delta$-groups so that if $\lambda_{i+1}-\lambda_i < \delta$, they are in the same $\delta$-group; if not, they are in different $\delta$-groups. We write $\lambda_i\equiv\lambda_j$ if $\lambda_i,\lambda_j$ are in the same $\delta$ group and $\lambda_i \not\equiv \lambda_j$ if they are not.   We note $(\widetilde{\Lambda}^{1/2}_rR-R\Lambda^{1/2}_r)_{ij} = (\strut \widetilde{\lambda}_i^{1/2}-\lambda_j^{1/2})R_{ij}$. We will show that when $\lambda_i,\lambda_j$ are in the same $\delta$ group, $\strut \widetilde{\lambda}_i^{1/2}-\lambda_j^{1/2}$ is small, and when they are in different $\delta$ groups, $|R_{ij}|$ is small, so that $(\widetilde{\Lambda}^{1/2}_rR-R\Lambda^{1/2}_r)_{ij}$ is always small. 
	First we relate $\widetilde{\lambda}_i^{1/2}$ to $\widetilde{\lambda}_i$ by
	\begin{align}
	\label{equ:eigval_relation}
	|\widetilde{\lambda}_i^{1/2}-\lambda_j^{1/2}| 
	&= \left| \frac{\widetilde{\lambda}_i - \lambda_i}{\widetilde{\lambda}_i^{1/2}+\lambda_i^{1/2}} + \frac{\lambda_i-\lambda_j}{\lambda_i^{1/2}+\lambda_j^{1/2}}\right| 
	\leq  \frac{\norm{\error}_2}{\lambda_r^{1/2}} + \frac{|\lambda_i-\lambda_j|}{\lambda_i^{1/2}+\lambda_j^{1/2}} \, .
	\end{align}
	Thus, when $\lambda_i,\lambda_j$ are in the same $\delta$ group, since $|R_{ij}|\leq 1$, we can bound
	\begin{align*}
	|(\widetilde{\Lambda}^{1/2}_rR-R\Lambda^{1/2}_r)_{ij}| &= |\strut \widetilde{\lambda}_i^{1/2}-\lambda_j^{1/2}| |R_{ij}| \\
	&\leq \frac{\norm{\error}_2}{\lambda_r^{1/2}} + \frac{r\delta}{\lambda_r^{1/2}}.
	\end{align*}
	To obtain a bound for $\lambda_i,\lambda_j$ in different $\delta$ groups, observe
	\begin{align*}
	R &= \widetilde{V}_r^{\T}V_r + \widetilde{V}_r^{\T}(\widetilde{V}_rR - V_r), \\
	|R_{ij}| &\leq |\langle v_i, \widetilde{v}_j \rangle | + \norm{\widetilde{V}_r^{\T}}_{\infty}\norm{\widetilde{V}_rR - V_r}_{\max} \\
	&\leq |\langle v_i, \widetilde{v}_j \rangle | +  N^{1/2}\norm{\widetilde{V}_r^{\T}}_{2}\norm{\widetilde{V}_rR - V_r}_{\max} \\
	&=|\langle v_i, \widetilde{v}_j \rangle | +  N^{1/2}\norm{\widetilde{V}_rR - V_r}_{\max}.
	\end{align*}
	When $i$ and $j$ correspond to different $\delta$-groups, $|\langle v_i, \widetilde{v}_j \rangle |  \leq\frac{\pi}{2}\norm{\error}_2/(\delta-2\norm{\error}_2)$ by {Theorem \ref{thm:EigspacePertBhatia} in Appendix \ref{app:EigPertResults}}.  Utilizing Equation \ref{equ:eigval_relation} and $\norm{\error}_2\lesssim \lambda_r$, we acquire
	\begin{align*}
	|(\widetilde{\Lambda}^{1/2}_rR-R\Lambda^{1/2}_r)_{ij}| 
	&=|\strut \widetilde{\lambda}_i^{1/2}-\lambda_j^{1/2}| |R_{ij}|\\
	&\leq \left(\frac{\norm{\error}_2}{\lambda_r^{1/2}} + \lambda_1^{1/2}\right)\left( \frac{\frac{\pi}{2}\norm{\error}_2}{\delta-2\norm{\error}_2} + N^{1/2}\norm{\widetilde{V}_rR - V_r}_{\max}\right) \\
	&\lesssim \lambda_1^{1/2} \left( \frac{\norm{\error}_2}{\delta-2\norm{\error}_2} + N^{1/2}\norm{\widetilde{V}_rR - V_r}_{\max}\right).
	\end{align*}
	We are free to choose $\delta$ to optimize our bounds, and we choose $\delta \propto (\lambda_1\lambda_r)^{1/4}\norm{\error}_2^{1/2}$. For this $\delta$, since $\norm{\error}_2 \lesssim \lambda_r$, we have $\delta-2\norm{\error}_2 \gtrsim \delta$, so that
	\begin{align*}
	|(\widetilde{\Lambda}^{1/2}_rR-R\Lambda^{1/2}_r)_{ij}| \lesssim \begin{cases} \norm{\error}_2/ \lambda_r^{1/2} + (\lambda_1/\lambda_r)^{1/4}\norm{\error}_2^{1/2} & (\lambda_i\equiv\lambda_j) \\[.2cm]
	(\lambda_1/\lambda_r)^{1/4}\norm{\error}_2^{1/2} + (\lambda_1 N)^{1/2}\norm{\widetilde{V}_rR - V_r}_{\max}  & (\lambda_i\not\equiv\lambda_j) \end{cases}.
	\end{align*}
	 Since $MM^\T$ is block constant, so are its eigenvectors on the blocks defined by the community labels, so that $\norm{v_j}_{\infty}\leq \nmin^{-1/2} = (\zeta N)^{-1/2}$ ($j=1,\ldots, s$), which gives $\norm{V_r}_{\max} \lesssim N^{-1/2}$. Since $\rho = \lambda_1/\lambda_r = O(1)$ by Condition \ref{cond:balance}, we obtain
	\begin{align*}
	\Gamma_2 &\lesssim \frac{\norm{\error}_2}{(\lambda_r N)^{1/2}} + \frac{\norm{\error}_2^{1/2}}{N^{1/2}}+\lambda_1^{1/2}\norm{\widetilde{V}_rR - V_r}_{\max}\, .
	\end{align*}
	Combining the bounds for $\Gamma_1$ and $\Gamma_2$, we thus obtain
	\begin{align*}
	\norm{\widetilde{V}_r\widetilde{\Lambda}^{1/2}_rR - V_r\Lambda^{1/2}_r}_{\max} &\lesssim \frac{\norm{\error}_2}{(\lambda_r N)^{1/2}} + \frac{\norm{\error}_2^{1/2}}{N^{1/2}}+\lambda_1^{1/2}\norm{\widetilde{V}_rR - V_r}_{\max} \, .
	\end{align*}
Since $\norm{\widetilde{V}_rR-V_r}_{\text{max}} = O\{\norm{\error}_{\infty}/( \lambda_r N^{1/2})\}$ by Theorem \ref{thm:eigenvec_pert}) and $\rho=O(1)$, we obtain
	\begin{align*}
	\norm{\widetilde{V}_r\widetilde{\Lambda}^{1/2}_rR - V_r\Lambda^{1/2}_r}_{\max} &\lesssim \frac{\norm{\error}_2}{(\lambda_r N)^{1/2}} + \frac{\norm{\error}_2^{1/2}}{N^{1/2}}+ \frac{\norm{\error}_{\infty}}{(\lambda_r N)^{1/2}}.
	\end{align*}
This completes the proof.
\end{proof}	

\begin{lem}
		\label{lem:embedding_pert_CommonCovCCP}
		Assume Conditions \ref{cond:balance}, \ref{cond:eigenvalue}, \ref{cond:CCP}, $\alpha = \big[\lambda_1/\{ \lambda_1+\mtr(\Sigma) \} \big]^{1/2}$, and $\norm{\error-\mtr(\Sigma)J}_2 \lesssim \lambda_r$. Then there exists a rotation matrix $R$ such that
		\begin{align*}
		\norm{\alpha\widetilde{V}_r\widetilde{\Lambda}^{1/2}_rR - V_r\Lambda^{1/2}_r}_{\max} &\lesssim \frac{\tau_{\rho}}{N^{1/2}} + \frac{ \norm{\error - \mtr(\Sigma)J}_2^{1/2}}{N^{1/2}}+ \frac{\norm{\error - \mtr(\Sigma)J}_2}{\lambda_r^{1/2}} \, ,
		\end{align*}
		where
		\[ \tau_\rho =\frac{\lambda_r^{1/2}(\rho-1)}{\big\{\lambda_1/\mtr(\Sigma)+1\big\} } \, . \] 
\end{lem}	

\begin{proof}
	A proof identical to that in Lemma \ref{lem:embedding_pert_general}, but with $\alpha \widetilde{\Lambda}_r^{1/2}$ replacing $\widetilde{\Lambda}_r^{1/2}$, gives
	\begin{align*}
	\norm{\alpha\widetilde{V}_r\widetilde{\Lambda}^{1/2}_rR - V_r\Lambda^{1/2}_r}_{\max}
	&\lesssim \underbrace{\alpha (\widetilde{\lambda}_1)^{1/2}\norm{\widetilde{V}_rR-V_r}_{\max}}_{:=\Gamma_1}
	+\underbrace{\norm{V_r}_{\max}\norm{\alpha\widetilde{\Lambda}^{1/2}_rR-R\Lambda^{1/2}_r}_{\max}}_{:=\Gamma_2}.
	\end{align*}
	
	
	To bound $\Gamma_1$ and $\Gamma_2$, we define the following quantities which are relevant for analyzing the perturbation $MM^\T \rightarrow (JX)(JX)^{\T}$: 
	\begin{align*}
	A &= MM^{\T}; \\
	\hat{A} &= MM^{\T} + \mtr(\Sigma)J; \\
	\widetilde{A} &= (JX)(JX)^{\T};\\
	\error &= \widetilde{A} - A; \\
	\error-\mtr(\Sigma)J &=  \widetilde{A} - \hat{A}. 
	\end{align*}
	The matrix $\widetilde{A} =A+\error$ is a perturbation of $A$; the quantity $\hat{A}$ will also be convenient in the proof, as we decompose the perturbation $A \rightarrow \widetilde{A}$ into two pieces: $A \rightarrow \hat{A}$ and $\hat{A} \rightarrow \widetilde{A}$.  We let $\{\hat{\lambda}_i\}_{i=1}^N$ denote the eigenvalues of $\hat{A}$. 
	
	We first relate the eigendecompositions of $A$ and $\hat{A}$. The $1$ vector is an eigenvector of $MM^{\T}$ with corresponding eigenvalue $\lambda=0$. Although $\lambda=0$ has multiplicity greater than 1, we choose a basis for its eigenspace so that $v_N=1$ and the remaining eigenvectors $v_1, \ldots, v_{N-1}$ are orthogonal to $1$. The matrix $J$ is a projection matrix which eliminates any component in the direction of $1$ but has no affect on any other direction. Thus for $i>N$
	\begin{align*}
	\hat{A}v_i = \{ MM^{\T}+\mtr(\Sigma)J \}v_i &= \lambda_iv_i+\mtr(\Sigma)v_i =  \{  \lambda_i+\mtr(\Sigma) \}v_i,
	\end{align*}
	and also
	\begin{align*}
	\hat{A}v_N&= \{ MM^{\T}+\mtr(\Sigma)J \}1= 0 = \lambda_Nv_N.
	\end{align*}
	Thus all eigenvectors of $A$ are still eigenvectors of $\hat{A}$, that is, we may choose an eigenbasis $\{\hat{v}_i\}_{i=1}^N$ for $\hat{A}$ such that $\hat{v}_i=v_i$ ($i=1,\ldots,N$) with $\hat{\lambda}_i = \lambda_i+\mtr(\Sigma)$ ($i=1,\ldots,N-1$) and $\hat{\lambda}_N=\lambda_N=0$. We let $\hat{V}_r =(\hat{v}_1,\ldots,\hat{v}_r)$ so that $V_r = \hat{V}_r$. 
	Furthermore, by Weyl's Inequality, $|\widetilde{\lambda}_i - \lambda_i| \leq \norm{\error}_2$ and $|\widetilde{\lambda}_i - \hat{\lambda}_i| \leq \norm{\error-\mtr(\Sigma)J}_2$ for all $i$.
	
	To obtain a bound for $\alpha \widetilde{\lambda}_1^{1/2}\norm{\widetilde{V}_rR-V_r}_{\max}$, note
	\begin{align*}
	\alpha\widetilde{\lambda}_1^{1/2} 
	&=\frac{\alpha^2\widetilde{\lambda}_1-\lambda_1}{\alpha\widetilde{\lambda}_1^{1/2}+\lambda_1^{1/2}}+\lambda_1^{1/2} 
	\leq \frac{|\alpha^2\widetilde{\lambda}_1-\lambda_1|}{\lambda_1^{1/2}}+\lambda_1^{1/2}.
	\end{align*}
	Utilizing $\alpha = \big[\lambda_1/\{ \lambda_1+\mtr(\Sigma) \} \big]^{1/2}$ and $|\widetilde{\lambda}_i - \hat{\lambda}_i| = |\widetilde{\lambda}_i - \lambda_i-\mtr(\Sigma)| \leq \norm{\error-\mtr(\Sigma)J}_2$, we have
	\begin{align*}
	\alpha\widetilde{\lambda}_1^{1/2} 
	&\leq \frac{\left| \lambda_1 \{ \lambda_1+\mtr(\Sigma)\pm\norm{\error-\mtr(\Sigma)J}_2 \}/\{ \lambda_1+\mtr(\Sigma)\} -\lambda_1\right|}{\lambda_1^{1/2}}+\lambda_1^{1/2} \\
	&= \frac{\left| \lambda_1\norm{\error-\mtr(\Sigma)J}_2 / \{ \lambda_1+\mtr(\Sigma)\}\right|}{\lambda_1^{1/2}}+\lambda_1^{1/2} \\ 
	&=\frac{\norm{\error-\mtr(\Sigma)J}_2}{\lambda_1^{1/2}}\left\{ 1+\frac{\mtr(\Sigma)}{\lambda_1}  \right\}^{-1}+\lambda_1^{1/2} \\
	&\leq \frac{\norm{\error-\mtr(\Sigma)J}_2}{\lambda_1^{1/2}}+\lambda_1^{1/2} \\
	&\lesssim \lambda_1^{1/2} \, ,
	\end{align*}
	since $\norm{\error-\mtr(\Sigma)J}_2 \lesssim \lambda_r \leq \lambda_1$. We thus obtain
	\begin{align*}
	\Gamma_1 = \alpha \widetilde{\lambda}_1^{1/2}\norm{\widetilde{V}_rR-V_r}_{\max} &\lesssim \lambda_1^{1/2}\norm{\widetilde{V}_rR-V_r}_{\max}\, .
	\end{align*}
	
	To bound $\Gamma_2$, we must control $\norm{\alpha\widetilde{\Lambda}^{1/2}_rR-R\Lambda^{1/2}_r}_{\max}$; note $(\alpha\widetilde{\Lambda}^{1/2}_rR-R\Lambda^{1/2}_r)_{ij} = (\alpha\strut \widetilde{\lambda}_i^{1/2}-\lambda_j^{1/2})R_{ij}$. As in the proof of Lemma \ref{lem:embedding_pert_general}, we partition the eigenvalues $\lambda_1, \ldots, \lambda_r$ into $\delta$-groups so that if $\lambda_{i+1}-\lambda_i < \delta$, they are in the same $\delta$-group, and if not, they are in different $\delta$-groups. We will show that when $\lambda_i,\lambda_j$ are in the same $\delta$ group, $\alpha\strut \widetilde{\lambda}_i^{1/2}-\lambda_j^{1/2}$ is small, and when they are in different $\delta$ groups, $|R_{ij}|$ is small, so that $(\alpha\widetilde{\Lambda}^{1/2}_rR-R\Lambda^{1/2}_r)_{ij}$ is always small. 
	First we relate $\widetilde{\lambda}_i^{1/2}$ to $\widetilde{\lambda}_i$ by
	\begin{align*}
	|\alpha\widetilde{\lambda}_i^{1/2}-\lambda_j^{1/2}| 
	&= \left| \frac{\alpha^2\widetilde{\lambda}_i - \lambda_i}{\alpha\widetilde{\lambda}_i^{1/2}+\lambda_i^{1/2}} + \frac{\lambda_i-\lambda_j}{\lambda_i^{1/2}+\lambda_j^{1/2}}\right| 
	\leq  \frac{|\alpha^2\widetilde{\lambda}_i - \lambda_i|}{\lambda_r^{1/2}} + \frac{|\lambda_i-\lambda_j|}{\lambda_r^{1/2}}.
	\end{align*}
	Utilizing $\alpha = \big[\lambda_1/\{ \lambda_1+\mtr(\Sigma) \} \big]^{1/2}$, we bound $|\alpha^2\widetilde{\lambda}_i - \lambda_i|$ by
	\begin{align*}
	|\alpha^2\widetilde{\lambda}_i - \lambda_i| &= \left| \left\{ \frac{\lambda_1}{\lambda_1+\mtr(\Sigma)}\right\}\widetilde{\lambda}_i-\lambda_i \right| \\
	&= \left| \left\{ \frac{\lambda_1}{\lambda_1+\mtr(\Sigma)}\right\}\{\lambda_i+\mtr(\Sigma)\pm\norm{\error-\mtr(\Sigma)J}_2\}-\lambda_i \right| \\
	&= \left| \frac{\mtr(\Sigma)(\lambda_1-\lambda_i) \pm \lambda_1 \norm{\error-\mtr(\Sigma)J}_2}{\lambda_1+\mtr(\Sigma)} \right| \\
	&\leq \frac{|\lambda_1 -\lambda_r|}{\big\{ \lambda_1/ \mtr(\Sigma)+1\big\}} + \norm{\error-\mtr(\Sigma)J}_2.
	\end{align*}
	Since  
	\[ \tau_\rho = \frac{\lambda_r^{1/2}(\rho-1)}{\big\{\lambda_1/\mtr(\Sigma)+1\big\} } = \frac{|\lambda_1 -\lambda_r|}{\lambda_r^{1/2} \big\{ \lambda_1/\mtr(\Sigma)+1\big\}}\, ,\] 
	we obtain
	\begin{align*}
	|\alpha\widetilde{\lambda}_i^{1/2}-\lambda_j^{1/2}| &\leq \frac{\lambda_r^{1/2}(\rho-1)}{\left\{\frac{\lambda_1}{\mtr(\Sigma)}+1\right\}} + \frac{\norm{\error-\mtr(\Sigma)J}_2}{\lambda_r^{1/2}} + \frac{|\lambda_i-\lambda_j|}{\lambda_r^{1/2}} \\
	&= \tau_\rho + \frac{\norm{\error-\mtr(\Sigma)J}_2}{\lambda_r^{1/2}} + \frac{|\lambda_i-\lambda_j|}{\lambda_r^{1/2}}.
	\end{align*}
	Thus, when $\lambda_i,\lambda_j$ are in the same $\delta$ group, since $|R_{ij}|\leq 1$, we can bound
	\begin{align*}
	|(\alpha\widetilde{\Lambda}^{1/2}_rR-R\Lambda^{1/2}_r)_{ij}| &= |\alpha\strut \widetilde{\lambda}_i^{1/2}-\lambda_j^{1/2}|(|R_{ij}|) \\
	&\leq  \tau_\rho + \frac{\norm{\error-\mtr(\Sigma)J}_2}{\lambda_r^{1/2}} + \frac{r\delta}{\lambda_r^{1/2}}.
	\end{align*}
	To obtain a bound for $\lambda_i,\lambda_j$ in different $\delta$ groups, recall from the proof of Lemma \ref{lem:embedding_pert_general} that
	\begin{align*}
	|R_{ij}| &\leq |\langle v_i, \widetilde{v}_j \rangle | +  N^{1/2}\norm{\widetilde{V}_rR - V_r}_{\max}.
	\end{align*}
	When $i$ and $j$ correspond to different $\delta$-groups, $|\langle v_i, \widetilde{v}_j \rangle | = |\langle \hat{v}_i, \widetilde{v}_j \rangle | \leq\frac{\pi}{2}\norm{\error-\mtr(\Sigma)J}_2/(\delta-2\norm{\error-\mtr(\Sigma)J}_2)$ {by Theorem \ref{thm:EigspacePertBhatia} in Appendix \ref{app:EigPertResults}}, and utilizing $\alpha \widetilde{\lambda}_1^{1/2} \lesssim \lambda_1^{1/2}$ as shown when bounding $\Gamma_1$, we obtain
	\begin{align*}
	|(\alpha\widetilde{\Lambda}^{1/2}_rR-R\Lambda^{1/2}_r)_{ij}| &\leq \left(|\alpha\widetilde{\lambda}_i^{1/2}| + |\lambda_j^{1/2}|\right)\left(\frac{\frac{\pi}{2}\norm{\error-\mtr(\Sigma)J}_2}{\delta-2\norm{\error-\mtr(\Sigma)J}_2} +N^{1/2}\norm{\widetilde{V}_rR - V_r}_{\max} \right) \\
	&\lesssim \lambda_1^{1/2}\left(\frac{\norm{\error-\mtr(\Sigma)J}_2}{\delta-2\norm{\error-\mtr(\Sigma)J}_2} +N^{1/2}\norm{\widetilde{V}_rR - V_r}_{\max} \right).
	\end{align*}  
	We are free to choose $\delta$ to optimize our bounds, and we choose $\delta = (\lambda_1\lambda_r)^{1/4}\norm{\error-\mtr(\Sigma)J}_2^{1/2}$. For this $\delta$, since $\norm{\error - \mtr(\Sigma)J}_2 \lesssim \lambda_r$, we have $\delta-2\norm{\error-\mtr(\Sigma)J}_2 \gtrsim \delta$, so that
	\begin{small}
	\begin{align*}
	|(\alpha\widetilde{\Lambda}^{1/2}_rR-R\Lambda^{1/2}_r)_{ij}| \lesssim \begin{cases} \tau_\rho+ \norm{\error-\mtr(\Sigma)J}_2/ \lambda_r^{1/2} + (\lambda_1/\lambda_r)^{1/4}\norm{\error-\mtr(\Sigma)J}_2^{1/2} & (\lambda_i\equiv\lambda_j) \\[.2cm]
	(\lambda_1/\lambda_r)^{1/4}\norm{\error-\mtr(\Sigma)J}_2^{1/2} + (\lambda_1 N)^{1/2}\norm{\widetilde{V}_rR - V_r}_{\max}  & (\lambda_i\not\equiv\lambda_j) \end{cases},
	\end{align*}
	\end{small}
	where $\lambda_i\equiv\lambda_j$ if $\lambda_i,\lambda_j$ are in the same $\delta$ group and $\lambda_i \not\equiv \lambda_j$ if they are not.
	
	Since as in the proof of Lemma \ref{lem:embedding_pert_general}, $\norm{V_r}_{\max} \lesssim N^{-1/2}$ and $\rho = O(1)$, we have
	\begin{align*}
	\Gamma_2 &\lesssim \frac{\tau_{\rho}}{N^{1/2}} + \frac{\norm{\error - \mtr(\Sigma)J}_2}{(\lambda_r N)^{1/2}} + \frac{ \norm{\error - \mtr(\Sigma)J}_2^{1/2}}{N^{1/2}}+ \lambda_1^{1/2} \norm{\widetilde{V}_rR - V_r}_{\max} 
	\end{align*}
	Combining our bounds for $\Gamma_1$ and $\Gamma_2$ thus gives
	\begin{align}
	&\norm{\alpha\widetilde{V}_r\widetilde{\Lambda}^{1/2}_rR - V_r\Lambda^{1/2}_r}_{\max} \nonumber \\
	&\qquad \lesssim \frac{\tau_{\rho}}{N^{1/2}} + \frac{\norm{\error - \mtr(\Sigma)J}_2}{(\lambda_r N)^{1/2}} + \frac{ \norm{\error - \mtr(\Sigma)J}_2^{1/2}}{N^{1/2}}+ \lambda_1^{1/2} \norm{\widetilde{V}_rR - V_r}_{\max} \, .
	\label{equ:embedding_bound_Vmax2}
	\end{align}
	To bound $\norm{\widetilde{V}_rR - V_r}_{\max}$, we note
	\begin{align*}
	\norm{\widetilde{V}_rR-V_r}_{\max} &= \norm{\widetilde{V}_rR-\hat{V}_r}_{\max} \leq \norm{\widetilde{V}_rR-\hat{V}_r}_\textnormal{F} 
	\end{align*} and by Theorem \ref{thm:DavisKahan}  {in Appendix \ref{app:EigPertResults}}  there exists an orthogonal matrix $R$ such that
	\begin{align*}
	\norm{\widetilde{V}_rR-\hat{V}_r}_\textnormal{F} &\leq \frac{2^{3/2}r^{1/2}\norm{\widetilde{A}-\hat{A}}_2}{ \hat{\lambda}_r-\hat{\lambda}_{r+1}} \\
	&= \frac{2^{3/2}r^{1/2}\norm{\widetilde{A}-\hat{A}}_2}{ \lambda_r+\mtr(\Sigma)-\{\lambda_{r+1}+\mtr(\Sigma)\}} \\
	&= \frac{2^{3/2}r^{1/2}\norm{\error-\mtr(\Sigma)J}_2}{ \lambda_r-\lambda_{r+1}} \\
	&= \frac{2^{5/2}r^{1/2}\norm{\error-\mtr(\Sigma)J}_2}{ \lambda_r}
	\end{align*}
	where the next to last line follows since $\lambda_r-\lambda_{r+1} \geq \lambda_r/2$. Utilizing the above bound in Inequality \ref{equ:embedding_bound_Vmax2} and $\rho=O(1)$, we obtain
	\begin{align*}
	\norm{\alpha\widetilde{V}_r\widetilde{\Lambda}^{1/2}_rR - V_r\Lambda^{1/2}_r}_{\max} &\lesssim \frac{\tau_{\rho}}{N^{1/2}} + \frac{ \norm{\error - \mtr(\Sigma)J}_2^{1/2}}{N^{1/2}}+ \frac{\norm{\error - \mtr(\Sigma)J}_2}{\lambda_r^{1/2}}
	\end{align*}
	and the lemma is proved.

\end{proof}

\section{Relating Perfect Geometric Representation with Embedding Perturbation}
\label{app:RelateEmbeddingAndRecovery}

\begin{lem}
	\label{lem:RelateER_EmbeddingPert}
	Assume Condition \ref{cond:eigenvalue} and the model in Section \ref{sec:ModelFormulation}. If there exists a constant $\alpha>0$ and rotation matrix $R$ such that
	\begin{align}
	\label{equ:exact_recovery_cond2}
	\mudiff > 12 \sqrt{r}\norm{V_r\Lambda^{1/2}_r - \alpha\widetilde{V}_r\widetilde{\Lambda}^{1/2}_rR}_{\max}\, ,
	\end{align}
	then the rank $r$ CMDS embedding $\widetilde{V}_r\widetilde{\Lambda}^{1/2}_r$ is a perfect geometric representation of the labels.
\end{lem}

\begin{proof}
Recall $s=\text{rank}(MM^\T)$ and $r \leq s$ by Condition \ref{cond:eigenvalue}; because $M$ is centered, $1 \leq s \leq k-1$. Let $V_i$ denote the $N \times i$ matrix whose columns are the top $i$ eigenvectors of $MM^\T$, and $\Lambda_i$ ($i=1,\ldots, s$) the $i \times i$ diagonal matrix of associated eigenvalues. Let $\widetilde{V}_i, \widetilde{\Lambda}_i$ denote the same quantities for the multidimensional scaling matrix $B$. Then the rows of $V_s\Lambda^{1/2}_s$ are the coordinates of the means $\{\mu_{\ell_i}^\T\}_{i=1}^N$ obtained from a full rank CMDS embedding of $M$, and the rows of $\widetilde{V}_r\widetilde{\Lambda}^{1/2}_r$ are the coordinates of the points $\{x_i^\T\}_{i=1}^N$ obtained from a rank $r$ CMDS embedding of $X$. We compare the clustering obtained from $V_s\Lambda^{1/2}_s$, which is perfect, to that obtained from $\widetilde{V}_r\widetilde{\Lambda}^{1/2}_r$, which is noisy. 
	Since the output of a clustering algorithm remains the same under a fixed scaling and rotation of the data, 
	$V_s\Lambda^{1/2}_s$ can in fact be compared with $\alpha\widetilde{V}_r\widetilde{\Lambda}^{1/2}_rR$, where $\alpha$ is any constant and $R$ is any $r\times r$ rotation matrix. Note $(V_s\Lambda^{1/2}_s)_{i\cdot} \in \mathbb{R}^s$ and $(\alpha V_r\Lambda^{1/2}_rR)_{i\cdot}\in \mathbb{R}^r$, but the coordinates $(\alpha V_r\Lambda^{1/2}_rR)_{i\cdot}$ are simply extended to coordinates in $\mathbb{R}^s$ by adding zeros in the remaining $s-r$ dimensions. Let $\{\phi_i^\T\}_{i=1}^N$ denote the rows of $V_s\Lambda^{1/2}_s$ and $\{\widetilde{\phi}_i^\T\}_{i=1}^N$ denote the rows of $\alpha V_r\Lambda^{1/2}_rR$, augmented with zeros in $s-r$ dimensions. 
	
	Because $M$ has rank $s$ and zero mean, the pairwise distances are perfectly preserved by the CMDS embedding, and $\norm{\mu_{\ell_i}^\T-\mu_{\ell_j}^\T}_2 = \norm{\phi_i^\T-\phi_j^\T}_2$. Let
	\begin{align*}
	d_{\text{in}}=d_{\text{in}}(\alpha\widetilde{V}_r\widetilde{\Lambda}_r^{1/2}R, \ell) &= \max_{i,j, \ell_i=\ell_j} \norm{\widetilde{\phi}_i^\T-\widetilde{\phi}_j^\T}_2, \\
	d_{\text{btw}}=d_{\text{btw}}(\alpha\widetilde{V}_r\widetilde{\Lambda}_r^{1/2}R, \ell) &= \min_{i,j, \ell_i\ne\ell_j} \norm{\widetilde{\phi}_i^\T-\widetilde{\phi}_j^\T}_2
	\end{align*}
	be the maximal within and minimal between community distances respectively. Recall that to obtain a perfect geometric representation of the data, and thus exactly recover the true labels via a simple clustering algorithm, we must have $2d_{\text{in}} < d_{\text{btw}}$. 
	
	Since
	\begin{align*}
	\norm{\widetilde{\phi}_i^\T-\widetilde{\phi}_j^\T}_2 &\leq \norm{\widetilde{\phi}_i^\T-\phi_i^\T}_2+\norm{\widetilde{\phi}_j^\T-\phi_j^\T}_2+\norm{\phi_i^\T-\phi_j^\T}_2 \\
	&= \norm{\widetilde{\phi}_i^\T-\phi_i^\T}_2+\norm{\widetilde{\phi}_j^\T-\phi_j^\T}_2+\norm{\mu_{\ell_i}^\T-\mu_{\ell_j}^\T}_2,
	\end{align*}
	we have
	\begin{align*}
	d_{\text{in}} &\leq 2 \max_i \norm{\widetilde{\phi}_i^\T-\phi_i^\T}_2.
	\end{align*}
	By a similar argument, since 
	\begin{align*}
	\norm{\phi_i^\T-\phi_j^\T}_2 &\leq \norm{\widetilde{\phi}_i^\T-\phi_i^\T}_2+\norm{\widetilde{\phi}_j^\T-\phi_j^\T}_2+\norm{\widetilde{\phi}_i^\T-\widetilde{\phi}_j^\T}_2 \\
	\norm{\widetilde{\phi}_i^\T-\widetilde{\phi}_j^\T}_2 &\geq \norm{\mu_{\ell_i}^\T-\mu_{\ell_j}^\T}_2- \norm{\widetilde{\phi}_i^\T-\phi_i^\T}_2-\norm{\widetilde{\phi}_j^\T-\phi_j^\T}_2,
	\end{align*}
	we have
	\begin{align*}
	d_{\text{btw}} &\geq \mudiff - 2 \max_i \norm{\widetilde{\phi}_i^\T-\phi_i^\T}_2.
	\end{align*}
	Thus, $2d_{\text{in}} < d_{\text{btw}}$ is guaranteed whenever
	\begin{align}
	\label{equ:exact_recovery_cond1}
	\mu_{\text{diff}} &> 6 \max_i \norm{\widetilde{\phi}_i^\T-\phi_i^\T}_2.
	\end{align}
	With a slight abuse of notation we have
	\begin{align*}
	\norm{\widetilde{\phi}_i^\T-\phi_i^\T}_2 &=  \norm{(V_s\Lambda^{1/2}_s)_{i\cdot} - \alpha(\widetilde{V}_r\widetilde{\Lambda}^{1/2}_rR)_{i\cdot}}_2 \\
	&\leq \norm{(V_s\Lambda^{1/2}_s)_{i\cdot} - (V_r\Lambda^{1/2}_r)_{i\cdot} }_{2} + 
	\norm{(V_r\Lambda^{1/2}_r)_{i\cdot} - \alpha(\widetilde{V}_r\widetilde{\Lambda}^{1/2}_rR)_{i\cdot}}_{2}, \\
	&\leq \norm{(V_s\Lambda^{1/2}_s)_{i\cdot} - (V_r\Lambda^{1/2}_r)_{i\cdot} }_{2} + 
	\sqrt{r}\norm{V_r\Lambda^{1/2}_r - \alpha\widetilde{V}_r\widetilde{\Lambda}^{1/2}_rR}_{\max},
	\end{align*}
	where we assume $r$-dimensional rows have been augmented with zeros to become $s$-dimensional rows when needed. 
	The first term represents the error incurred from discarding the small eigenvalues and eigenvectors of $MM^\T$, but this error is easily controlled when $\lambda_{r+1}$ is small. Since $MM^\T$ is block constant, so are its eigenvectors on the blocks defined by the community labels, so that $\norm{v_j}_{\infty}\leq \nmin^{-1/2}$ ($j=1,\ldots, s$). Since the first $r$ dimensions match exactly, we thus have for all $i$
	\begin{align*}
	\norm{(V_s\Lambda^{1/2}_s)_{i\cdot} - (V_r\Lambda^{1/2}_r)_{i\cdot} }_{2} &\leq \norm{[0, \ldots, 0, \lambda_{r+1}^{1/2}v_{r+1}(i), \ldots, \lambda_s^{1/2}v_s(i)]}_{2} \leq \frac{(s-r)^{1/2}\lambda_{r+1}^{1/2}}{\nmin^{1/2}}.
	\end{align*}
	Since $\mudiff > 12(s-r)^{1/2}\lambda_{r+1}^{1/2}\nmin^{-1/2}$ by Condition \ref{cond:eigenvalue}, $ \mudiff > 12 \max_i \norm{(V_s\Lambda^{1/2}_s)_{i\cdot} - (V_r\Lambda^{1/2}_r)_{i\cdot} }_{2}$. Thus to guarantee Inequality \ref{equ:exact_recovery_cond1}, it is sufficient for
	\begin{align*}
	\mudiff &> 12 \sqrt{r}\norm{V_r\Lambda^{1/2}_r - \alpha\widetilde{V}_r\widetilde{\Lambda}^{1/2}_rR}_{\max}.
	\end{align*}
	
\end{proof}

\section{Clustering Results}
\label{app:clustering}

\subsection{$k$-means}
\label{app:kMeans}
Given $k$ initial centroids $c_1, \ldots, c_k$,  the standard $k$-means algorithm known as Loyd's algorithm partitions a data set by iteratively assigning each data point to its closest centroid and then recomputing the centroids until convergence. The goal is to recover a partition $S= \{S_1,\ldots, S_k\}$ which minimizes the $k$-means objective
\begin{align}
\label{equ:kmeans_obj}
G(S) &= \sum_{m = 1}^k \frac{1}{2|S_m|} \sum_{x_i, x_j\in S_m} \norm{x_i - x_j}_2^2.
\end{align}
The following lemma guarantees that any perfect geometric representation $S$ as defined in Section \ref{sec:ModelFormulation} is a local minimum of the $k$-means objective, so that Loyd's algorithm will converge to $S$ when appropriately initialized. Indeed, if Loyd's algorithm is run with furthest point initialization, the initial centroids will consist of one point from each set $S_i$, and Loyd's algorithm will perfectly recover the parition $S$. The furthest point initialization procedure randomly selects a data point as the first centroid $c_1$ and 
then for $i=2,\ldots, k$ iteratively defines centroid $c_{i+1}$ by selecting the data point $x \in X\setminus \{c_1,\ldots,c_i\}$ which maximizes $\min_{1\leq j\leq i}\norm{x-c_j}_2$.


\begin{lemma}
	\label{lem:kmeansObjective}
	Let $S = \{S_1,\ldots, S_k\}$ be a partition of the data into clusters with $S_m = \{x_i: \ell_i = m \}$, and let
	\begin{align*}
	d_{\textnormal{in}} (X, \ell) &= \max_{m \in [k] } \max_{x_i , x_j \in S_m} \norm{x_i- x_j}_2\ , \qquad d_{\textnormal{btw}} (X, \ell) = \min_{m \ne l} \min_{x_i\in S_m, x_j\in S_l} \norm{x_i -x_j}_2
	\end{align*}
	be the maximal within cluster and minimal between cluster distances respectively. Then if $d_{\textnormal{btw}} > (3/2)^{1/2} d_{\textnormal{in}}$, the partition $S$ is a local minimum of the $k$-means objective.
\end{lemma}

\begin{proof}
	We use contradiction to prove this lemma. Suppose not. Then there is a point reassignment that increases the $k$-means objective $G(S)$ defined in Equation \ref{equ:kmeans_obj}. Specifically, there exists a point $x \in S_m$ such that moving it to another cluster $S_l$ decreases $G(S)$. Let $\hat{S} = \{\hat{ S}_1,\ldots, \hat{S}_k\}$ denote the same clustering as $S$ except point $x$ is 
	moved from $S_m$ to $S_l$, i.e.,  $\hat{S}_m = S_m \setminus x$ and $\hat{S}_l = S_l \cup x$. We have $G(\hat{S}) < G(S)$ and 
	\begin{align*}
	G(S) - G(\hat{S}) = &  \sum_{g = 1}^k \frac{1}{2|S_g| } \sum_{x_i, x_j\in S_g} \norm{x_i - x_j}_2^2 -  \sum_{ g = 1}^k \frac{1}{2|\hat{S}_g |} \sum_{ x_i', x_j' \in \hat{S}_g} \norm{x_i' - x_j'}_2^2 \\
	= & \left( \frac{1}{2|S_m|} \sum_{x_i, x_j\in S_m} \norm{x_i - x_j}_2^2 -  \frac{1}{2|\hat{S}_m|} \sum_{x_i', x_j'\in \hat{S}_m} \norm{x_i' - x_j'}_2^2\right) \\
	& + \left( \frac{1}{2|S_l |} \sum_{x_i, x_j\in S_l} \norm{x_i - x_j}_2^2 - \frac{1}{2|\hat{S}_l|} \sum_{x_i', x_j' \in \hat{S}_l} \norm{x_i' -  x_j'}_2^2\right) \\
	& \coloneqq (A) + (B),
	\end{align*}
	where 
	\begin{align*}
	(A) &= \frac{1}{2|S_ m|} \sum_{x_i , x_j \in S_m} \norm{x_i - x_j}_2^2 -  \frac{1}{2|S_m|} \sum_{x_i', x_j'\in \{  S_m\setminus x \}} \frac{|S_m|}{|S_m| -1}\norm{x_i' - x_j'}_2^2 \\
	&= \frac{1}{2|S_m|} \sum_{x_i, x_j\in \{ S_m\setminus x \}} \norm{y-z}_2^2 +\frac{1}{|S_m|} \sum_{x_i \in \{ S_m\setminus x\}} \norm{ x_i - x}_2^2  \\
	&\qquad -  \frac{1}{2|S_m|}\sum_{ x_i , x_j \in \{S_m \setminus x\}} \frac{|S_m|}{|S_m| -1} \norm{ x_i - x_j}_2^2 \\
	&= \frac{1}{2|S_m|} \sum_{  x_i , x_j\in \{ S_m\setminus x \}} \norm{ x_i - x_j}_2^2\left(1-\frac{|S_m|}{|S_m|-1}\right) +\frac{1}{|S_m|} \sum_{x_i \in  \{ S_m\setminus x \}} \norm{ x_i -x}_2^2; \\
	(B) &=  \frac{1}{2|S_l |} \sum_{ x_i , x_j\in S_l} \norm{ x_i - x_j}_2^2 - \frac{1}{2|S_l|} \sum_{ x_i , x_j\in \{ S_l\cup x \}} \frac{|S_l|}{|S_l| + 1}\norm{  x_i - x_j}_2^2 \\
	&=  \frac{1}{2|S_l |} \sum_{ x_i , x_j\in S_l} \norm{ x_i - x_j}_2^2 - \frac{1}{2|S_l|} \sum_{ x_i , x_j\in S_l} \frac{|S_l |}{|S_l | + 1}\norm{ x_i - x_j}_2^2 - \frac{1}{|S_l|} \sum_{y \in S_l} \norm{ x_i - x_j}_2^2 \\
	&= \frac{1}{2|S_l|} \sum_{ x_i , x_j\in S_l}\norm{ x_i - x_j}_2^2\left(1 - \frac{|S_l|}{|S_l|+1}\right)- \frac{1}{|S_l|} \sum_{x_i \in S_l}
	\norm{x_i-x}_2^2.
	\end{align*}
	Thus, we have 
	\begin{align*}
	G(S) &- G(\hat{S}) \\
	&= \frac{1}{2|S_m|} \sum_{x_i, x_j\in \{ S_m \setminus x \}} \norm{x_i - x_j}_2^2\left(1-\frac{|S_m|}{|S_m|-1}\right) +\frac{1}{|S_m|} \sum_{ \{ x_i \in S_m\setminus x \}} \norm{x_i-x}_2^2 \\
	&\qquad +  \frac{1}{2|S_l|} \sum_{ x_i , x_j\in S_l}\norm{ x_i - x_j}_2^2\left(1 - \frac{|S_l|}{|S_l|+1}\right)- \frac{1}{|S_l|} \sum_{x_i \in S_j} \norm{x_i-x}_2^2 \\
	&\leq \frac{1}{|S_m|} \sum_{x_i \in \{  S_m\setminus x \}} d_{\text{in}}^2 +  \frac{1}{2|S_l|} \sum_{ x_i , x_j\in S_l}d_{\text{in}}^2 \left(\frac{1}{|S_l|+1}\right)- \frac{1}{|S_l|} \sum_{x_i \in S_l} d_{\text{btw}}^2 \\
	&= \frac{(|S_m|-1)}{|S_m|}d_{\text{in}}^2 + \frac{|S_l|(|S_l|-1)}{2|S_l|(|S_l|+1)}d_{\text{in}}^2 - \frac{|S_l|}{|S_l|}d_{\text{btw}}^2 \\
	&= \frac{(|S_m|-1)}{|S_m|}d_{\text{in}}^2 + \frac{(|S_l|-1)}{2(|S_l|+1)}d_{\text{in}}^2 - d_{\text{btw}}^2 \\
	&\leq d_{\text{in}}^2+\frac{1}{2}d_{\text{in}}^2 - d_{\text{btw}}^2 \\
	&= \frac{3}{2}d_{\text{in}}^2 - d_{\text{btw}}^2.
	\end{align*}	
	Since moving the point $x$ decreases the $k$-means objective, we have:
	\begin{align*}
	0 &< G(S) - G(\hat{S}) \leq \frac{3}{2}d_{\text{in}}^2 - d_{\text{btw}}^2,
	\end{align*}
	which gives $d_{\text{btw}} \leq (3/2)^{1/2} d_{\text{in}}$, which contradicts the assumptions of Lemma \ref{lem:kmeansObjective}. Thus the partition $S$ corresponds to a local minimum of the $k$-means objective.	
\end{proof}

\subsection{Hierarchical Clustering}

{Agglomerative hierarchical clustering algorithms initially view each data point as a singleton cluster and then clusters are iteratively merged by choosing clusters $A, B$ which minimizes some linkage function $d(A,B) $. Some common linkage functions are
\begin{align*}
d_{\text{SL}}(A,B) &= \min_{x_i\in A, x_j\in B} \norm{x_i-x_j}_2 \, , \qquad &\text{(single linkage)} \\
d_{\text{CL}}(A,B) &= \max_{x_i\in A, x_j\in B} \norm{x_i-x_j}_2\, , \qquad &\text{(complete linkage)} \\
d_{\text{AL}}(A,B) &= \frac{1}{|A|\cdot|B|} \sum_{x_i \in A, x_j \in B} \norm{x_i-x_j}_2\, , \qquad &\text{(average linkage)} \\
d_{\text{EN}}(A,B) &=  \frac{2}{|A|\cdot|B|} \sum_{x_i \in A, x_j \in B} \norm{x_i-x_j}_2 - \frac{1}{|A|^2} \sum_{x_i \in A, x_j \in A} \norm{x_i-x_j}_2 & \\ 
&\quad - \frac{1}{|B|^2} \sum_{x_i \in B, x_j \in B} \norm{x_i-x_j}_2 \, .\qquad &\text{(minimum energy)} 
\end{align*}
Iterating this procedure until all points are joined in a single cluster creates the hierarchical structure known as a dendrogram.
The following lemma is a direct consequence of the agglomerative hierarchical  clustering procedure and the definition of a perfect geometric representation $S$; it ensures that $S$ can be recovered by selecting the hierarchical clustering which corresponds to $k$ clusters.
\begin{lem}
	Let $S = \{S_1,\ldots, S_k\}$ be a partition of the data into clusters with $S_m = \{x_i: \ell_i = m \}$, and let
	\begin{align*}
	d_{\textnormal{in}} (X, \ell) &= \max_{m \in [k] } \max_{x_i , x_j \in S_m} \norm{x_i- x_j}_2\ , \qquad d_{\textnormal{btw}} (X, \ell) = \min_{m \ne l} \min_{x_i\in S_m, x_j\in S_l} \norm{x_i -x_j}_2 \  .
	\end{align*}
	If $d_{\textnormal{in}} (X, \ell) < 2 d_{\textnormal{btw}}(X, \ell) $, then $S$ is recovered from hierarchical clustering with single linkage, complete linkage, average linkage, and minimum energy by selecing the dendrogram level corresponding to $k$ clusters. 
\end{lem}
\begin{proof}
	We must prove that all points in each $S_m$ will be merged before any subsets of $S_m, S_l$ for $m\ne l$ are merged. Suppose not. Then there exists a first time in the algorithm where $A\subset S_m, A \ne S_m$ and $B\subset S_l$ are merged. Because this is the first time of such an event and $A \ne S_m$, there exists another candidate cluster $C \subset A^C \cap S_m$ which was not selected. We verify that $d(A,C) < d(A,B)$ which leads to a contradiction. Note since $d_{\textnormal{in}} < 2 d_{\textnormal{btw}}$, there exists an $\epsilon$ satisfying $d_{\textnormal{in}} (X, \ell) < \epsilon < 2 d_{\textnormal{btw}}$. \\	
	Since all distances between points in $A$ and $C$ are less then $\epsilon$, so are the maximum, minimum, and average distances between points in these sets, so $d(A,C) < \epsilon$ for $d=d_{\text{SL}}, d_{\text{CL}}, d_{\text{AL}}$. Similarly, since all distances between points in $A$ and $B$ are larger than $2\epsilon$, so are the maximum, minimum, and average distances between points in these sets, so $d(A,B)>2\epsilon$ for $d=d_{\text{SL}}, d_{\text{CL}}, d_{\text{AL}}$. We conclude $d(A,C) < d(A,B)$ for $d=d_{\text{SL}}, d_{\text{CL}}, d_{\text{AL}}$ which is a contradiction.
	For $d_{\text{EN}}$, we note
	\begin{align*}
	d_{\text{EN}}(A,B) &> 4\epsilon - \epsilon - \epsilon = 2\epsilon \\
	d_{\text{EN}}(A,C) &<2\epsilon -0 -0 = 2\epsilon \, ,
	\end{align*}
	so once again $d_{\text{EN}}(A,C) < d_{\text{EN}}(A,B)$, which is a contradiction.
\end{proof}

\section{Proofs of Random Matrix Lemmas}
\label{app:RMTLemmas}

This section states and proves Lemmas \ref{lem:RMTSpectralControl}, \ref{lem:RMTInfinityControl}, and  \ref{lem:RMTCenteredSpectralControl}, three random matrix lemmas which are needed in the proofs of Theorems \ref{thm:eigenvec_pert}, \ref{thm:MDS_embedding}, \ref{thm:MainResult}, and \ref{thm:MainTheorem_HighD}. Throughout this section we utilize $JM=M$ and $M^\T J=M^\T$ (since the rows of $M$ are centered), so that the matrix $P$ defined in Equation \ref{equ:error_matrix} has the form 
\begin{align*}
\error &= 
JMH^TJ+JHM^TJ+JHH^TJ = MH^TJ+JHM^T+JHH^TJ.
\end{align*}
Recall $\mumax = \max_{1\leq i\leq N} \norm{M_{i\cdot}}_2$, $J=I_N - (1/N)11^{\T}$ and $\gamma=d/N$. 
\begin{lemma}
	\label{lem:RMTSpectralControl}
	Let $M$ be a deterministic $N$ by $d$ matrix with zero row mean. Let $H$ be a random $N$ by $d$ matrix with independent sub-Gaussian rows, and $\sigmax = \max_{1\leq i\leq N} \norm{H_{i\cdot}}_{\psi_2} $ its maximal sub-Gaussian row norm. Then
	$
	\error = \{J(M+H)\}\{J(M+H)\}^\T - MM^\T
	$
 satisfies
	\begin{align*}
	\frac{1}{N}\norm{P}_2 &\lesssim \sigmax\mumax(1+\sqrt{\gamma})+ \sigmax^2(1+\gamma)
	\end{align*}
with probability at least $1-O(N^{-1})$.
\end{lemma}

\begin{lemma}
	\label{lem:RMTInfinityControl}
	Let $M$ be a deterministic $N$ by $d$ matrix with zero row mean. Let $H$ by a random $N$ by $d$ matrix with independent sub-Gaussian rows, and $\sigmax = \max_{1\leq i\leq N} \norm{H_{i\cdot}}_{\psi_2} $ its maximal sub-Gaussian row norm. Then
	$
	\error = \{J(M+H)\}\{J(M+H)\}^\T - MM^\T
	$
    satisfies
	\begin{align*}
	\frac{1}{N}\norm{P}_{\infty} &\lesssim \sigmax\mumax (\log N)^{1/2} + \sigmax^2 (\sqrt{d} \log N+ \gamma)
	\end{align*}
	with probability at least $1-O(N^{-1})$.
\end{lemma}

\begin{lemma}\label{lem:RMTCenteredSpectralControl}
	Let $M$ be a deterministic $N$ by $d$ matrix with zero row mean. Let $H$ by a random $N$ by $d$ matrix with independent sub-Gaussian rows, and $\sigmax \geq \max_{1\leq i\leq N} \norm{H_{i\cdot}}_{\psi_2} $ an upper bound on the maximal sub-Gaussian row norm. In addition, assume the rows of $H$ satisfy the convex concentration property with constant $\sigmax$ as defined in Condition \ref{cond:CCP}. Then
	$
	\error = \{J(M+H)\}\{J(M+H)\}^\T - MM^\T
	$
	satisfies
	\begin{align*}
	\frac{1}{N}\norm{\error-\mtr(\Sigma)J}_2 &\lesssim \sigmax\mu_{\text{max}} + \sigmax^2\left(\gamma^{1/2} \vee 1\right)
	\end{align*}
	with probability at least $1-O(N^{-1})$.
\end{lemma}

\subsection{Proof of Lemma \ref{lem:RMTSpectralControl}}

\begin{proof}
	Since $J^2=J$ and $\norm{J}_2=1$, we have
	\begin{align*}
	\error &= MH^{\T}J+JHM^{\T}+JHH^{\T}J\\
	\norm{\error}_2 &\leq \norm{MH^{\T}}_2\norm{J}_2+\norm{J}_2\norm{HM^{\T}}_2+\norm{J}_2\norm{HH^{\T}}_2\norm{J}_2 \\ 
	&\leq 2\norm{MH^{\T}}_{2}+\norm{HH^{\T}}_{2}
	\end{align*}
	We now bound each term. Standard random matrix results (for example see \cite{vershynin2012high}) give $\norm{H}_2 \lesssim \sigmax(\sqrt{N}+\sqrt{d})$ with probability at least $1-O(N^{-1})$ so that utilizing $\norm{M}_F \leq \sqrt{N}\mumax$ we obtain
	\begin{align*}
	\norm{MH^\T}_2 &\leq \norm{M}_2 \norm{H}_2 \leq \norm{M}_F \norm{H}_2 \lesssim \mumax \sigmax(N+\sqrt{Nd})
	\end{align*}
	and
	\begin{align*}
	\norm{HH^\T}_2 &\lesssim \sigmax^2(N+d)\, .
	\end{align*}
	Dividing by $N$ proves the lemma.
\end{proof}	

\subsection{Proof of Lemma \ref{lem:RMTInfinityControl}}

\begin{proof}	
	Since $J^2=J$, we decompose $\error$ as
	\begin{align*}
	\error &= J(MH^{\T}+HM^{\T})J+J (HH^\T -\text{diag}(HH^{\T})) J + J \text{diag}(HH^{\T}) J  \\
	\norm{\error}_\infty &\leq \norm{J(MH^{\T}+HM^{\T})J}_{\infty}+\norm{J(HH^{\T}-\text{diag}(HH^{\T}))J}_\infty  +\norm{J\text{diag}(HH^{\T})J}_\infty \\
	& \coloneqq \text{(I)} + \text{(II)} + \text{(III)} \,,
	\end{align*}
	and proceed to bound each term. Recall that for any matrix product $AB$, $\norm{AB}_{\max} \leq \norm{A}_{\infty}\norm{B}_{\max}$ and also $\norm{AB}_{\max} \leq \norm{A}_{\max}\norm{B}_{1}$. Since $\norm{J}_\infty=\norm{J}_1 = 2(N-1)/N \leq 2$, we can conclude that for any matrix $A$, $\norm{JAJ}_{\text{max}} \leq 4\norm{A}_{\text{max}}$, a property we will apply throughout the proof. For a random variable $X$, $\norm{X}_{\psi_2}$, $\norm{X}_{\psi_1}$ denote the sub-Gaussian, subexponential norm of $X$. \vspace{.2cm}\\
	\noindent\textit{Bounding  \text{(I)}}\vspace{.2cm}\\
	Let $m_i$ denote the $i^{\text{th}}$ row of $M$ and $h_j$ the $j^{\text{th}}$ row of $H$. Because $m_i/\norm{m_i}$ is a unit vector, $\langle m_i/\norm{m_i}, h_j \rangle$ is sub-Gaussian with norm bounded by $\sigmax$. Thus
	\begin{align*}
	\Prob(|\langle m_i/\norm{m_i}, h_j \rangle| > t) &\leq C_2 \exp(-t^2/(C_1\sigmax^2)) \\
	\Prob(|\langle m_i, h_j \rangle| > t\mumax) &\leq C_2 \exp(-t^2/(C_1\sigmax^2)) \, .
	\end{align*}
	Taking a union bound over all $i,j$ gives
	\begin{align*}
	\Prob(\norm{MH^\T}_{\max}> t\mumax) &\leq C_2 N^2\exp(-t^2/(C_1\sigmax^2))\, .
	\end{align*}
	Choosing $t \sim \sigma(\log N)^{1/2}$ gives that with probability at least $1-O(N^{-1})$
	\begin{align*}
	\norm{MH^\T}_{\max} &\lesssim \sigmax\mumax(\log N)^{1/2}
	\end{align*}
	so that 
	\begin{align*}
	\text{(I)} &\leq  N \norm{J(MH^{\T}+HM^{\T})J}_{\max} \leq 4N \norm{(MH^{\T}+HM^{\T})}_{\max} \lesssim \sigmax\mumax N (\log N)^{1/2}\, .
	\end{align*}
	
	\noindent\textit{Bounding (II)}\vspace{.2cm}\\	
	We consider $(HH^\T)_{i,j} = \langle h_i, h_j \rangle$ for $i\ne j$. By Lemma \ref{lem:sub_exp_inner_prod}, which is stated and proved at the end of this subsection,
	\begin{align*}
	\Prob(|\langle h_i, h_j \rangle| > t) &\leq C_2\exp(-C_1t/(\sqrt{d}\sigmax^2))
	\end{align*}
	which gives
	\begin{align*}
	\Prob(\max_{i\ne j}|\langle h_i, h_j \rangle| > t) &\leq C_2N^2\exp(-C_1t/(\sqrt{d}\sigmax^2)) \, .
	\end{align*}
	Choosing $t \sim \sigmax^2 \sqrt{d} \log N$ then gives that with probability at least $1-O(N^{-1})$
	\begin{align*}
	\norm{HH^\T - \text{diag}(HH^\T)}_{\max} &\lesssim \sigmax^2 \sqrt{d} \log N
	\end{align*}
	so that
	\begin{align*}
	\text{(II)} &= \norm{J( HH^{\T}-\text{diag}(HH^{\T}) ) J}_{\infty} \\
	&\leq N\norm{J(HH^{\T}-\text{diag}(HH^{\T}))J}_{\max} \\ 
	&\leq 4N\norm{HH^{\T}-\text{diag}(HH^{\T})}_{\max} \\ 
	&\lesssim \sigmax^2 \sqrt{d} N \log N \, .
	\end{align*}
	
	\noindent\textit{Bounding (III)}\vspace{.2cm}\\	
	We now consider $(HH^\T)_{ii} = \norm{h_i}^2$. Since $h_i$ is a sub-Gaussian random vector with $\norm{h_i}_{\psi_2} \leq \sigmax$, using one-step discretization, we shall obtain  with probability $1-\delta$  
	\begin{align*}
	\norm{h_i}_2 &\lesssim \sigmax\sqrt{d} + \sigmax\sqrt{\log(1/\delta)}, \\
	\norm{h_i}_2^2 &\lesssim \sigmax^2 (d + \log(1/\delta) ).
	\end{align*}
	Thus with probability $1-N\delta$
	\begin{align*}
	\max_{i} \norm{h_i}_2^2 &\lesssim \sigmax^2 (d + \log(1/\delta) )\, ,
	\end{align*}
	so choosing $\delta = 1/N^2$ we obtain that with probability at least $1-O(N^{-1})$
	\begin{align*}
	\norm{\text{diag}(HH^\T)}_{\infty} &= \norm{\text{diag}(HH^\T)}_{\max} \lesssim  \sigmax^2 (d + \log N ),
	\end{align*}
	which gives
	\begin{align*}
	\text{(III)} &= \norm{J\text{diag}(HH^{\T})J}_\infty \\
	&\leq \norm{J}_\infty\norm{\text{diag}(HH^{\T})}_\infty\norm{J}_\infty \\
	&\leq 4  \norm{\text{diag}(HH^{\T})}_{\max} \\
	&\lesssim \sigmax^2 (d + \log N ) \, .
	\end{align*}
	Putting the three pieces together, we have
	\begin{align*}
	\norm{\error}_{\infty} &\lesssim \sigmax\mumax N (\log N)^{1/2} + \sigmax^2 (\sqrt{d} N \log N+ d)\, .
	\end{align*}
	Dividing by $N$ proves the lemma.
\end{proof}	

The following lemma is a standard result in measure concentration. We collect it here for completeness. 
\begin{lemma}
	\label{lem:sub_exp_inner_prod}
	Let $X, Y \in \R^d$ be independent sub-Gaussian random vectors with $\max\{\norm{X}_{\psi_2}, \norm{Y}_{\psi_2}\} \leq \sigmax$. Then $\langle X, Y \rangle$ is sub-exponential and $\|\langle Y, X \rangle\|_{\psi_1} \leq C \sqrt{d}\sigmax^2$ for some absolute constant $C$.
\end{lemma}
\begin{proof}
We use $C$ to denote a generic constant, which may change line to line.  We proceed by bounding the moment generating function $\Ex e^{\lambda \langle Y, X \rangle}$.  
	Since $\|Y\|_{\psi_2} \leq \sigmax$, $\| \langle Y, u \rangle \|_{\psi_2} \leq \sigmax$ for any unit vector $u$. Thus by Proposition 2.5.2 in \cite{vershynin2018high}, for any unit vector $u$ we have
	\[ \Ex e^{\lambda\langle Y, u \rangle} \leq e^{C\sigmax^2\lambda^2} \quad \forall \lambda \, ,\]
	or equivalently,
	\begin{equation}
	\label{equ:subgauss_prop}
	\Ex e^{\langle Y, a \rangle} \leq e^{C\sigmax^2\|a\|^2} \quad \forall a \, ,
	\end{equation}
	and Inequality \ref{equ:subgauss_prop} holds for $X$ as well. Letting $a=\lambda X$ in the above, since $X$ and $Y$ are independent, we obtain
	\begin{equation}
	\label{equ:2}
	\Ex\, e^{\lambda \langle Y, X \rangle} \leq \Ex\, e^{C\sigmax^2\lambda^2\|X\|^2} \, .
	\end{equation}
	We have (deterministically)
	\begin{align*}
	e^{\frac{r\|X\|^2}{2}} &= \frac{1}{(2\pi r)^{d/2}} \int_{\R^d} e^{-\frac{\|a\|^2}{2r}+\langle a, X \rangle }\ da \, .
	\end{align*}
	Taking expectation,
	\begin{align*}
	\Ex e^{\frac{r\|X\|^2}{2}} &=\frac{1}{(2\pi r)^{d/2}} \int_{\R^d} e^{-\frac{\|a\|^2}{2r}}\Ex e^{\langle a, X \rangle }\ da \\
	&\leq \frac{1}{(2\pi r)^{d/2}} \int_{\R^d} e^{-\frac{\|a\|^2}{2r}+C\sigmax^2\|a\|^2 }\ da \qquad \text{(by Equation \ref{equ:subgauss_prop})} \\
	&= \frac{1}{(2\pi r)^{d/2}} \int_{\R^d} e^{-\|a\|^2\left(\frac{1}{2r}-C\sigmax^2\right) }\ da \\
	&= \frac{1}{(1-2rC\sigmax^2)^{d/2}}.
	\end{align*}
	We choose $r = 2C\sigmax^2\lambda^2$ in the previous calculation, so that we can bound the right hand side of Inequality \ref{equ:2} and obtain
	\begin{align*}
	\Ex\, e^{\lambda \langle Y, X \rangle} &\leq \Ex\, e^{C\sigmax^2\lambda^2\|X\|^2} \\
	&\leq \frac{1}{(1-2(2C\sigmax^2\lambda^2)C\sigmax^2)^{d/2}} \\
	&= \frac{1}{(1-C\sigmax^4\lambda^2)^{d/2}}.
	\end{align*}
 Now since $\frac{1}{1-x} \leq e^{2x}$ for $|x|\leq 0.8$, we have
	\begin{align*}
	\frac{1}{1-C\sigmax^4\lambda^2} &\leq e^{2C\sigmax^4\lambda^2} \\
	\Ex\, e^{\lambda \langle Y, X \rangle} \leq  \frac{1}{(1-C\sigmax^4\lambda^2)^{d/2}} &\leq e^{Cd\sigmax^4\lambda^2}.
	\end{align*}
	Since the above holds for all $\lambda^2 \leq \frac{1}{Cd\sigmax^2}$ (assuming $d>1$), by Property (e) of subexponential distributions in \cite{vershynin2018high}, $\langle Y, X \rangle$ is thus a subexponential random variable with subexponential norm $\|\langle Y, X \rangle\|_{\psi_1} \leq C \sqrt{d}\sigmax^2$. 
\end{proof}	

\subsection{Proof of Lemma \ref{lem:RMTCenteredSpectralControl}}

\begin{proof}
Since $J^2=J$ and $\norm{J}_2=1$, we have
\begin{align*}
\error - \mtr({\Sigma})J &= MH^{\T}J+JHM^{\T}+J(HH^{\T} - \mtr(\Sigma)I_N)J\\
\norm{\error-\mtr(\Sigma)J}_2 &\leq \norm{MH^{\T}}_2\norm{J}_2+\norm{J}_2\norm{HM^{\T}}_2+\norm{J}_2\norm{HH^{\T}-\mtr(\Sigma)I_N}_2\norm{J}_2 \\ 
&\leq 2\norm{MH^{\T}}_{2}+\norm{HH^{\T}-\mtr(\Sigma)I_N}_{2} \\
&= \text{(I)} + \text{(II)}.
\end{align*}
We now bound each term.\vspace{.2cm}\\
\noindent\textit{Bounding \text{(I)}}\vspace{.2cm} \\ 
We first note that $MH^{\T}$ has independent columns, since $(MH^{\T})_{\cdot i} = M\eta_i$. Secondly, we will show that Condition \ref{cond:CCP} guarantees that $\norm{M\eta_i}_{\psi_2} \leq \norm{M}_2 \sigmax$. The latter holds if $ \norm{\langle M\eta_i, u\rangle}_{\psi_2} \leq \norm{M}_2 \sigmax$ for all unit vectors $u$. Since $\norm{M}_2^{-1}\langle M\eta_i, u\rangle$ is a centered, 1-Lipschitz, convex function of $\eta_i$, by Condition \ref{cond:CCP},
\begin{align*}
\Prob\left( |\langle M\eta_i, u\rangle| > t\norm{M}_2\right) &\leq 2\exp(-t^2/\sigmax^2) \\
\implies \Prob\left( |\langle M\eta_i, u\rangle| > \widetilde{t}\right) &\leq 2\exp(-\widetilde{t}^2/\norm{M}_2^2\sigmax^2) \\
\implies \norm{\langle M\eta_i, u\rangle}_{\psi_2} &\leq \norm{M}_2 \sigmax \, .
\end{align*}
Thus standard random matrix results guarantee that $\norm{MH^{\T}}_{2} \lesssim \norm{M}_2 \sigmax \sqrt{N}$ with probability $1-O(N^{-1})$. Since $\norm{M}_2 \leq \norm{M}_F \leq \mumax\sqrt{N}$, one obtains
\begin{align*}
\norm{MH^{\T}}_{2} &\lesssim \sigmax\mumax N \, .
\end{align*} 
\noindent\textit{Bounding \text{(II)}}\vspace{.2cm} \\ 
First let $x \in S^{N-1}$ and consider
\begin{align*}
x^T (HH^T - \text{tr}(\Sigma)I_N) x &= \langle H^Tx, H^Tx \rangle - \text{tr}(\Sigma) = \norm{h}_2^2 -  \text{tr}(\Sigma) 
\end{align*}
where $h = H^Tx$. Note that since the left hand side is mean zero, the right hand side is also mean zero, so $\Ex(\norm{h}_2^2)=\text{tr}(\Sigma)$. By Lemma C.11 in \cite{kasiviswanathan2019restricted}, since the columns of $H^{\T}$ are independent vectors satisfying the convex concentration property with constant $\sigmax$, the random vector $h = H^Tx$ also satisfies the convex concentration property with constant $O(\sigmax)$. We can thus apply the following Hanson-Wright type inequality, which appears as Theorem 2.5 in \cite{adamczak2015note}. 
\begin{theorem}
	Let $h$ be a mean zero random vector in $\mathbb{R}^N$. If $h$ has the convex concentration property with constant $K$, then for any $N \times N$ matrix $A$ and every $t > 0$,
	 \begin{align*}
	\Prob\left( \left| h^{\T}A h  - \Ex(h^{\T}A h)\right|>t \right) &\leq 2 \exp\left(-\frac{1}{C}\min\left\{\frac{t^2}{2K^4 \norm{A}_{HS}^2} , \frac{t}{ K^2\norm{A}_{2}}\right\}\right)
	\end{align*}
	for some universal constant $C$.
\end{theorem}
Applying the above theorem with $A=I_d$, $K=O(\sigmax)$ gives
\begin{align*}
	\Prob\left( \left|  \norm{h}_2^2 - \Ex \norm{h}_2^2 \right|>t \right) &\leq 2 \exp\left(-\frac{1}{C}\min\left\{\frac{t^2}{2d\sigmax^4} , \frac{t}{ \sigmax^2}\right\}\right) \, .
\end{align*}
Thus when $t \lesssim d\sigmax^2$, one obtains
\begin{align*}
\Prob(| \norm{h}_2^2 - \Ex \norm{h}_2^2| > t) \leq 2\exp(-ct/ \sigmax^2)\\
\Prob(| x^T (HH^T - \text{tr}(\Sigma)I_N) x| > t) \leq 2\exp(-ct/ \sigmax^2)
\end{align*}
for some constant $c$.  Similarly, when $t\gtrsim d\sigmax^2$ one obtains
\begin{align*}
\Prob(| \norm{h}_2^2 - \Ex \norm{h}_2^2| > t) \leq 2\exp(-ct^2/(\sqrt{d}\sigmax^2)^2),\\
\Prob(| x^T (HH^T - \text{tr}(\Sigma)I_N) x| > t) \leq 2\exp(-ct^2/( \sqrt{d}\sigmax^2)^2).  
\end{align*}
We now bound $\norm{HH^{\T}-\mtr(\Sigma)I_N}_{2}$ with a covering argument. Let $\mathcal{N}$ denote an $\epsilon$ net of $S^{N-1}$. Since $HH^T - \text{tr}(\Sigma)I_N$ is symmetric, we have (see Example 4.4.3 in \cite{vershynin2018high})
\begin{align*}
\norm{HH^T - \text{tr}(\Sigma)I_N}_2 &\leq \frac{1}{1-2\epsilon}\,  \sup_{x\in \mathcal{N}} \, |x^T(HH^T - \text{tr}(\Sigma)I_N)x| \quad,\quad \epsilon \in [0,1/2) \, .
\end{align*}
Also by Corollary 4.2.13 in \cite{vershynin2018high}, we have $|\mathcal{N}| \leq \left(\frac{2}{\epsilon}+1\right)^{N}$. Choosing for example $\epsilon = 1/4$ so that $|\mathcal{N}| \leq 9^N$, we have
\begin{align*}
\norm{HH^T - \text{tr}(\Sigma)I_N}_2 &\leq 2 \sup_{x \in \mathcal{N}}|x^T(HH^T - \text{tr}(\Sigma)I_N)x|
\end{align*}
If $t\gtrsim  d\sigmax^2$, a union bound yields that
\begin{align*}
\Prob\left( \sup_{x \in \mathcal{N}}| x^T (HH^T - \text{tr}(\Sigma)I_N) x| > t\right) \leq (9^N)2\exp(-ct/(\sigmax^2)) \\
\Prob\left(\norm{HH^T - \text{tr}(\Sigma)I_N}_2> 2t\right) \leq (9^N)2\exp(-ct/(\sigmax^2)).
\end{align*}
If $t\lesssim  d\sigmax^2$, a union bound yields
\begin{align*}
\Prob\left( \sup_{x \in \mathcal{N}}| x^T (HH^T - \text{tr}(\Sigma)I_N) x| > t\right) \leq (9^N)2\exp(-ct^2/(\sqrt{d}\sigmax^2)^2) \\
\Prob\left(\norm{HH^T - \text{tr}(\Sigma)I_N}_2> 2t\right) \leq (9^N)2\exp(-ct^2/(\sqrt{d}\sigmax^2)^2).
\end{align*}
When $N\geq d$, taking  $t \sim N\sigmax^2$ yields
\begin{align*}
\norm{HH^T - \text{tr}(\Sigma)I_N}_2 &\lesssim N\sigmax^2
\end{align*}
with high probability at least $1-O(N^{-1})$. 
When $N \leq d$, taking $t \sim \sqrt{Nd}\sigmax^2$ yields
\begin{align*}
\norm{HH^T - \text{tr}(\Sigma)I_N}_2 &\lesssim \sqrt{Nd}\sigmax^2
\end{align*}
with high probability at least $1-O(N^{-1})$. 
Therefore with high probability 
\begin{align*}
\norm{HH^T - \text{tr}(\Sigma)I_N}_2 &\lesssim \sigmax^2 \sqrt{N} (\sqrt{d}\vee \sqrt{N})\, .
\end{align*}
In sum, we have
\begin{align*}
\text{(I)} \lesssim \sigmax\mu_{\max}N, \quad \text{(II)} \lesssim \sigmax^2N^{1/2}\left(d^{1/2} \vee N^{1/2}\right).
\end{align*}
We can thus bound the spectral norm of $\{\error-\mtr(\Sigma)J\}/N$ by
\begin{align*}
\frac{\norm{\error-\mtr(\Sigma)J}_2}{N} &\lesssim \sigmax\mu_{\text{max}} + \sigmax^2\left(\gamma^{1/2} \vee 1\right).
\end{align*}
\end{proof}

\section{Eigenspace Perturbation Results}\label{app:EigPertResults}
{This appendix collects some known results for bounding eigenspace perturbations.} For {Theorems \ref{thm:LowRankInfinityEigPert} and \ref{thm:DavisKahan}}, $A, \widetilde{A} \in \mathbb{R}^{N\times N}$ are {real,} symmetric with eigenvalues $\lambda_1\geq \ldots \geq \lambda_N$, $\widetilde{\lambda}_1 \geq \ldots \geq \widetilde{\lambda}_N$ and associated eigenvectors $v_1,\ldots,v_N$ and $\widetilde{v}_1,\ldots,\widetilde{v}_N$ respectively. {Theorem} \ref{thm:LowRankInfinityEigPert}  is a recent result by \cite{fan2018eigenvector}.
\begin{theorem}
	\label{thm:LowRankInfinityEigPert}
	Let $A_r = \sum_{i=1}^r \lambda_i v_iv_i^{\T}$ and $\kappa = (N/r) \max_i \sum_{j=1}^r v^2_{j}(i)$. Suppose $|\lambda_r|-\epsilon = \Omega(r^3\kappa^2\norm{A - \widetilde{A}}_{\infty})$, where $\epsilon = \norm{A - A_r}_{\infty}$. Then, there exists an orthogonal matrix $R \in \mathbb{R}^{r\times r}$ such that
	\[ 
	\norm{\widetilde{V}_rR-V_r}_{\textnormal{max}} = O\left\{\frac{r^{5/2}\kappa^2\norm{A - \widetilde{A}}_{\infty}}{(|\lambda_r|-\epsilon)N^{1/2}}\right\} , 
	\]
	where $V_r=(v_1,\ldots,v_r),\widetilde{V}_r=(\widetilde{v}_1,\ldots\widetilde{v}_r)$.
\end{theorem}

{Theorem \ref{thm:DavisKahan}}  is a variation of the standard Davis-Kahan Theorem developed by \citet{yu2014useful}. 
\begin{theorem}
	\label{thm:DavisKahan}
	Fix $1\leq m\leq s \leq N$ and assume that $\min(\lambda_{m-1}-\lambda_m,\lambda_s-\lambda_{s+1})>0$, where we define $\lambda_0 = \infty$ and $\lambda_{N+1}=-\infty$. Let $r=s-m+1$, and $V_r = (v_m,v_{m+1},\ldots,v_s) \in \mathbb{R}^{N \times r}$, $\widetilde{V}_r =( \widetilde{v}_m,\widetilde{v}_{m+1},\ldots,\widetilde{v}_s) \in \mathbb{R}^{N \times r}$. We have
	\begin{align*}
	\norm{\sin \Theta(\widetilde{V}_r,V_r)}_\textnormal{F} &\leq \frac{2\min(r^{1/2}\norm{\widetilde{A}-A}_2, \norm{\widetilde{A}-A}_\textnormal{F})}{\min(\lambda_{m-1}-\lambda_m, \lambda_s-\lambda_{s+1})},
	\end{align*}
	where $\Theta(\widetilde{V}_r,V_r)$ denotes the $r$ by $r$ diagonal matrix whose $j^{\text{th}}$ diagonal entry is the $j^{\text{th}}$ principal angle between $\widetilde{V}_r,V_r$, and $\sin \Theta(\widetilde{V}_r,V_r)$ is defined entry-wise.
	Moreover, there exists an orthogonal matrix $R \in \mathbb{R}^{N \times N}$ such that
	\begin{align*}
	\norm{\widetilde{V}_rR-V_r}_\textnormal{F} \leq \frac{2^{3/2}\min(r^{1/2}\norm{\widetilde{A}-A}_2, \norm{\widetilde{A}-A}_\textnormal{F})}{\min(\lambda_{m-1}-\lambda_m, \lambda_s-\lambda_{s+1})}.
	\end{align*}
\end{theorem}

{Finally, Theorem \ref{thm:EigspacePertBhatia} appears as Theorem VII.3.2 in \cite{bhatia2013matrix}.} {We use  $\opnorm{\cdot}{}$ to denote a unitarily invariant matrix norm, i.e. one satisfying $\opnorm{UAV}{}=\opnorm{A}{}$ for all unitary matrices $U,V$. Let $P_A(S)$ denote the orthogonal projection onto the subspace spanned by the eigenvectors of $A$ with eigenvalues in $S_1$.}

\begin{theorem}
	\label{thm:EigspacePertBhatia}
	Let $A$ and $\widetilde{A}$ be Hermitian operators, and let $S_1, S_2$ be any two subsets of $\mathbb{R}$ such that $\dist(S_1,S_2) ={ \inf_{x \in S_1} \inf_{y \in S_2} |x - y| } =\delta>0$. Let $E=P_A(S_1), F=P_{\widetilde{A}}(S_2)$. Then, for every unitarily invariant norm, 
	\begin{align*}
	\opnorm{EF}{} &\leq \frac{\pi}{2\delta} \opnorm{E(A-\widetilde{A})F}{} \leq \frac{\pi}{2\delta} \opnorm{A-\widetilde{A}}{}.
	\end{align*}
\end{theorem}

\bibliographystyle{chicago}
\bibliography{MDS}

\end{document}